\documentclass[12pt]{article}

\usepackage[utf8]{inputenc}
\usepackage[T1]{fontenc}
\usepackage{lmodern}
\usepackage[english]{babel}
\usepackage{amsmath,mathtools,amsthm,amssymb,amsfonts}
\usepackage{color}
\usepackage{comment}

\usepackage{hyperref}
\usepackage{mathrsfs}
\usepackage{graphicx}
\usepackage{enumitem}
\usepackage{tikz} 
\usepackage{esint}
\usepackage{pgfplots}
\usepackage{esvect}
\usepackage{pdfpages}
\usepackage{nicefrac}
\usepackage{leftidx, tensor}
\usepackage{cleveref}
\usepackage{accents}
\usetikzlibrary{shapes,arrows}
\usetikzlibrary{calc}
\usetikzlibrary{shapes.geometric}
\usetikzlibrary{shapes.arrows}
\usetikzlibrary{arrows.meta}
\usetikzlibrary{decorations.markings}
\usetikzlibrary{fit}
\usetikzlibrary{patterns}
\usetikzlibrary{hobby}
\usepgfplotslibrary{patchplots}
\pgfplotsset{compat=1.11}

\renewcommand{\thefootnote}{\fnsymbol{footnote}}
\newcommand{\definedas}{\mathrel{\raise.095ex\hbox{\rm :}\mkern-5.2mu=}}

\newcommand{\R}{\mathbb{R}}
\newcommand{\N}{\mathbb{N}}

\newcommand{\Sbb}{\mathbb{S}}

\renewcommand{\d}{\,\mathrm{d}}
\newcommand{\graph}{\,\mathrm{graph}}

\newcommand{\Hess}{\mathrm{Hess}}

\newcommand{\ul}[1]{\underline{#1}}

\newcommand{\btr}[1]{\left\vert#1\right\vert}
\newcommand{\newbtr}[1]{\vert#1\vert}
\newcommand{\norm}[1]{\btr{\btr{#1}}}

\newcommand{\spann}[1]{\left\langle#1\right\rangle}

\newcommand{\Ric}{\mathrm{Ric}}

\newcommand{\Rm}{\mathrm{Rm}}

\newcommand{\two}{\operatorname{II}}
\newcommand{\tr}{\text{tr}}
\newcommand{\dive}{\operatorname{div}}

\newcommand{\newnorm}[1]{\vert\vert#1\vert\vert}

\theoremstyle{plain}
\newtheorem{thm}{Theorem}[section]
\newtheorem{prop}[thm]{Proposition}

\newtheorem{lem}[thm]{Lemma}

\theoremstyle{definition}
\newtheorem{defi}[thm]{Definition}

\newtheorem{bem}[thm]{Remark}
\newtheorem{kor}[thm]{Corollary}

\begin{document}
		\begin{center}
			{\LARGE {A De\,Lellis--M\"uller type estimate on the Minkowski lightcone}\par}
		\end{center}
		\vspace{0.5cm}
		\begin{center}
			{\large Markus Wolff\footnote[2]{wolff@math.uni-tuebingen.de}}\\
			\vspace{0.4cm}
			{\large Department of Mathematics}\\
			{\large Eberhard Karls Universit\"at T\"ubingen}
		\end{center}
		\vspace{0.4cm}
		\begin{abstract}
			We prove an analogue statement to an estimate by De\,Lellis--M\"uller in $\R^3$ on the standard Minkowski lightcone. More precisely, we show that under some additional assumptions, any spacelike cross section of the standard lightcone is $W^{2,2}$-close to a round surface provided the trace-free part of a scalar second fundamental form $A$ is sufficiently small in $L^2$. To determine the correct intrinsically round cross section of reference, we define an associated $4$-vector, which transforms equivariantly under Lorentz transformations in the restricted Lorentz group. A key step in the proof consists of a geometric, scaling invariant estimate, and we give two different proofs. One utilizes a recent characterization of singularity models of null mean curvature flow along the standard lightcone by the author, while the other is heavily inspired by an almost-Schur lemma by De\,Lellis--Topping.
		\end{abstract}
		\renewcommand{\thefootnote}{\arabic{footnote}}
		\setcounter{footnote}{0}
	\section{Introduction}
	It is a well-known fact that any closed, embedded hypersurface in $\R^3$ with vanishing trace-free part $\accentset{\circ}{h}\equiv 0$ of the second fundamental form has constant mean curvature $H$, and is thus a round sphere by \cite{aleks}. De\,Lellis--M\"uller \cite{delellismueller} have proven a quantitative statement motivated by this in the sense that if $\accentset{\circ}{h}$ is sufficiently small in $L^2$, then the surface $\Sigma$ under consideration is $W^{1,2}$-close to a round sphere $\Sbb^2_r(\vec{a})$, where $r$ and $\vec{a}$ denote the area radius and Euclidean center of $\Sigma$, respectively.
	
	In a recent paper \cite{wolff1} the author introduced a \emph{scalar second fundamental form} A of spacelike cross sections $\Sigma$ of the standard Minkowski lightcone, such that the trace-free part $\accentset{\circ}{A}$ of $A$ vanishes if and only if $\Sigma$ is a surface of \emph{constant spacetime mean curvature} (in the lightcone). We note that such an \emph{STCMC surface} of the Minkwoski lightcone arises as an intersection with an affine, spacelike hyperplane, and thus these surfaces are precisely round spheres up to a suitable Lorentz transformation in the restriced Lorentz group $\operatorname{SO}^+(1,3)$, which are directly related to a conformal diffeomorphism in the M\"obius group. More generally, any spacelike cross section $\Sigma$ of the standard Minkowski lightcone can be realized as a graph over the round sphere $\Sbb^2$ with a conformally round metric, where the conformal factor $\omega$ simultaneously prescribes the graphical height on the lightcone.
	
	Here, we provide a quantitative statement of the aforementioned observation by the author similar to the work of De\,Lellis--M\"uller. More precisely, we show under some additional assumptions that for any spacelike cross section $\Sigma$ of the lightcone such that the trace-free part $\accentset{\circ}{A}$ of $A$ is sufficiently small in $L^2(\Sigma)$, the conformal factor $\omega$ of $\Sigma$ is close to the conformal factor of an appropiate STCMC surface of reference in $W^{2,2}(\Sbb^2)$, see Theorem \ref{thm_delellismueller}. The STCMC surface of reference is uniquely determined by an associated $4$-vector $\textbf{Z}$ in the ambient Minkowski spacetime that is timelike and future-pointing for any spacelike cross section of the standard Minkowski lightcone. This associated $4$-vector is closely related to a notion of center (of mass) in asymptotically hyperbolic initial data sets, cf. Cederbaum--Cortier--Sakovich \cite{cederbaumcortiersakovich}. Moreover, its spacial components closely resemble a notion of center of mass in conformal geometry, cf. \cite[Definition 2.8]{klainermanszeftel}. We show that $\textbf{Z}$ transforms equivariantly under Lorentz transformations, see Proposition \ref{prop_Lorentztransform}.
	
	As in the work of De\,Lellis--Mueller \cite{delellismueller}, the proof of the statement consists of two steps:
	In the first step, we establish a geometric, scaling invariant estimate of the form
	\begin{align}\label{intro_tracefreeA}
	\btr{\Sigma}\int_\Sigma \btr{A-\frac{\fint\mathcal{H}^2}{2}\gamma}^2\le C\btr{\Sigma} \int_\Sigma\newbtr{\accentset{\circ}{A}}^2
	\end{align}
	for any spacelike cross section $\Sigma$ of the standard lightcone, where $C>0$ is a uniform constant. We present two proofs of this desired estimate which both require the initial assumption $\mathcal{H}^2\ge 0$. The \emph{spacetime mean curvature} $\mathcal{H}^2=\tr_\Sigma A$ corresponds to the Lorentzian length $\eta(\vec{\mathcal{H}},\vec{\mathcal{H}})$ of the mean curvature vector $\vec{\mathcal{H}}$ of $\Sigma$ in the ambient Minkowski spacetime $(\R^{1,3},\eta)$, and can hence be (locally) negative in general. Moreover, by the Gauss equation one finds that $\mathcal{H}^2=2\operatorname{R}$, where $\operatorname{R}$ denotes the scalar curvature of $\Sigma$, so $\mathcal{H}^2$ completely determines the intrinsic geometry of the conformally round spacelike cross section. We would like to emphasize that we state estimate Estimate \eqref{intro_tracefreeA} with respect to the scalar second fundamental form $A$ and not with respect to the full, vector-valued second fundamental form $\vec{\two}$ of $\Sigma$, because $\accentset{\circ}{\vec{\two}}(X,Y)$ turns out to be zero or null for any two tangent vector fields $X,Y$, i.e.,
	\[
		\newnorm{\accentset{\circ}{\vec{\two}}}^2=0
	\]
	for any spacelike cross section $\Sigma$, although $\accentset{\circ}{\vec{\two}}\not=0$ in general. On the other hand, we have $\newbtr{\accentset{\circ}{A}}=0$ if and only if $\mathcal{H}^2=\operatorname{const.}$ as stated above, so we may think of $A$ as a convenient scalar representative of $\vec{\two}$ in the context of such geometric estimates. Nontheless $A$ encodes the same geometric information for spacelike cross sections of the standard Minkowski lightcone. Although $A$ is defined in a frame-independent way, it can in fact be realised as a (a-priori frame-dependent) null second fundamental form of $\Sigma$ upon a particular choice of gauge. We note that up to a factor of $2$ this gauge can be obtained by solving inverse mean curvature flow as an ODE along the lightcone, cf. Sauter \cite{sauter}, and that along any spacelike cross section $\Sigma$, the null generator in this particular gauge corresponds to the position vector of points on $\Sigma$. We note that this particular choice of gauge and its related null second fundamental form, i.e., what we call the scalar second fundamental form $A$ here, have been independently utilized by Le to prove several geometric (in-)equalities for conformally round surfaces, see \cite{le1, le2, le3}.
	
	The first proof of Estimate \eqref{intro_tracefreeA} we present here relies on \emph{null mean curvature flow}, see Proposition \ref{thm_tracefreeA1}. This parabolic flow for spacelike cross sections of a null hypersurface has first been studied by Roesch--Scheuer \cite{roeschscheuer} to detect marginally outer trapped surfaces (MOTS), and in a recent paper \cite{wolff1} by the author in the case of the standard Minkowski lightcone. As no MOTS exists in the Minkowski spacetime, the flow develops singularities in finite time. A full classification of all singularity models as STCMC surfaces has been obtained by the author in \cite{wolff1} by showing that null mean curvature flow on the standard lightcone is equivalent to $2d$-Ricci flow on topological spheres and drawing upon a classical result first proven by Hamilton \cite{hamilton1}, see also \cite{chow1, bartzstruweye, struwe, andrewsbryan}. Additionally, the author gave a new proof of this result  in \cite{wolff1} under Hamilton's initial restriction, which translates to the additional assumption of $\mathcal{H}^2>0$ in this context. 
	
	Rather surprisingly, the parabolic flow does not provide the optimal estimate, which is instead obtained by the corresponding elliptic equation and using the Bochner formula. This second proof is motivated by an almost-Schur lemma by De\,Lellis--Topping \cite{Delellistopping}, see Proposition \ref{thm_tracefreeA3}, and indeed provides the optimal constant (Theorem \ref{thm_tracefreeA4}). Moreover, as $\mathcal{H}^2=2\operatorname{R}$ this result can be interpreted as a generalization of the work of De\,Lellis--Topping to $2$ dimensions. In fact, this proof directly extends to higher dimensions where both estimates are equivalent, see Section \ref{sec_discussion} for a brief discussion. Both proofs show that equality is achieved if and only if both sides of the inequality vanish identically.
	
	Upon rescaling by the area radius, the left-hand side of \eqref{intro_tracefreeA} provides a uniform control of the difference $\mathcal{K}-1$ in $L^2$ of the rescaled surface, where $\mathcal{K}$ denotes its Gauss curvature. Adapting an estimate by Shi--Wang--Wu \cite{shiwangwu} to our setting, we obtain a $W^{2,2}(\Sbb^2)$-estimate on the difference $\omega-1$ under some additional a-priori estimate and a suitable balancing condition, see Proposition \ref{prop_w22}. The required balancing condition is intricately connected to the associated $4$-vector $\textbf{Z}$, and can always be achieved (Proposition \ref{prop_Lorentztransform}). Finally, reversing the rescaling and the applied Lorentz transformation yields the desired estimate with the appropriate STCMC surface of reference. As a consequence, we obtain a criteria for sequences of spacelike cross sections to posses a limiting subsequence upon rescaling that converges to an STCMC surface, see Theorem \ref{thm_delellismueller2}
	\newline\\
	This paper is structured as follows:\newline
	In Section \ref{sec_prelim} we give the necessary prelimaries, and introduce the associated $4$-vector $\textbf{Z}$ in Section \ref{sec_as4vector}. In Section \ref{sec_estimate}, we provide the two proofs of estimate \eqref{intro_tracefreeA}. We provide the necessary elliptic estimate in Section \ref{sec_estimate}, and proof the main result Theorem \ref{thm_delellismueller}, and Theorem \ref{thm_delellismueller2} in Section \ref{sec_main}. We close with some comments in Section \ref{sec_discussion}.
		
		\subsection*{Acknowledgements.}
		I would like to express my sincere gratitude towards my PhD supervisors Carla Cederbaum and Gerhard Huisken for their continuing guidance and helpful discussions. Further, I want to thank Annachiara Piubello for pointing me towards some extremely helpful references regarding the contents of Section \ref{sec_elliptic}, and Rodrigo Avalos for pointing me towards the almost-Schur lemma by De\,Lellis--Topping.
		\setcounter{section}{1}
	\section{Preliminaries}\label{sec_prelim}
	\subsection{General setup}\label{subsec_prelim_general}
	Throughout this paper, $\R^{1,3}$ will always refer to the $1+3$-dimensional Minkowski spacetime $(\R^{1,3},\eta)$, where $\R^4$ is equipped with the flat metric $\eta$ of signature $(-+++)$. In polar coordinates, $\R^{1,3}$ is given as $\R\times(0,\infty)\times\Sbb^2$ with
	\[
	\eta=-\d t^2+ \d r^2+r^2\d\Omega^2,
	\]
	where $\d\Omega^2$ denotes the standard round metric on $\Sbb^2$. The standard lightcone centered at the origin in the Minkowski spacetime is given as the set
	\[
	C(0):=\{\btr{t}=r\},
	\]
	and we denote the components $C(0)_+:=C(0)\cap\{t\ge 0\}$, $C(0)_-=C(0)\cap\{t\le 0\}$ as the future-pointing and past-pointing standard lightcone (centered at the origin and with time-orientation induced by $\partial_t$), respectively.
	
	In the following, $(\Sigma,\gamma)$ will always denote a $2$-surface with Riemanian metric $\gamma$, and we are in particular interested in such surfaces arising as closed, orientable, spacelike codimension-$2$ surfaces in $\R^{1,3}$ (usually restricted to the lightcone). As usual, we define the vector valued second fundamental form $\vec\two$ of $\Sigma$ in $\R^{1,3}$ as 
	\begin{align*}
	\vec\two(V,W)=\left(\overline{\nabla}_VW\right)^\perp,
	\end{align*}
	where $\overline{\nabla}$ denotes the Levi-Civita connection on the Minkowski spacetime. The codimension-$2$ mean curvature vector of $\Sigma$ is then defined as $\vec{\mathcal{H}}=\tr_\gamma \vec\two$, where $\tr_\gamma$ denotes the metric trace on $\Sigma$ with respect to $\gamma$, and we define the spacetime mean curvature $\mathcal{H}^2$ of $\Sigma$ as the Lorentzain length of $\vec{\mathcal{H}}$, i.e., $\mathcal{H}^2:=\eta(\vec{\mathcal{H}},\vec{\mathcal{H}})$. Let $\{\ul{L},L\}$ be a null frame of $\Gamma(T^\perp\Sigma))$, i.e., $\eta(\ul{L},\ul{L})=\eta(L,L)=0$, with $\eta(\ul{L},L)=2$. Then $\vec{\two}$ and $\vec{\mathcal{H}}$ admit the decomposition 
	\begin{align}
	\begin{split}
		\vec{\two}&=-\frac{1}{2}\chi\ul{L}-\frac{1}{2}\ul{\chi}L,\\
		\vec{\mathcal{H}}&=-\frac{1}{2}\theta\ul{L}-\frac{1}{2}\ul{\theta}L.
	\end{split}
	\end{align}
	Here, the null fundamental forms $\ul{\chi}$ and $\chi$ with respect to $\ul{L}$ and $L$, respectively, are defined as
	\begin{align*}
	\ul{\chi}(V,W)&\definedas \spann{\overline{\nabla}_V\ul{L},W}=-\spann{\overline{\nabla}_VW,\ul{L}},\\
	{\chi}(V,W)&\definedas \spann{\overline{\nabla}_V{L},W}=-\spann{\overline{\nabla}_VW,L},
	\end{align*}
	for all tangent vector fields $V,W\in \Gamma(T\Sigma)$, and the null expansions $\ul{\theta}\definedas \tr_\gamma\ul{\chi}$ and $\theta\definedas \tr_\gamma{\chi}$ with respect to $\ul{L}$ and $L$, respectively.
	In particular
	\begin{align}
	\begin{split}\label{eq_secondffnulldecomp}
	\newbtr{\vec\two}^2&=\spann{\chi,\ul{\chi}},\\
	\mathcal{H}^2&=\ul{\theta}\theta,
	\end{split}
	\end{align}
	and if $\mathcal{H}^2$ is constant along $\Sigma$, i.e., $\vec{\mathcal{H}}$ has constant Lorentzian length along $\Sigma$, we call $\Sigma$ a surface of constant spacetime mean curvature (STCMC), cf. Cederbaum--Sakovich \cite{cederbaumsakovich}.
	Finally, we define the connection one-form $\zeta$ as
	\[
	\zeta(V)\definedas\frac{1}{2}\spann{\overline{\nabla}_V\ul{L},L}.
	\]
	
	We denote the Riemann curvature tensor, Ricci curvature and scalar curvature on  $(\Sigma,\gamma)$ by $\Rm$, $\Ric$, $\operatorname{R}$, respectively, where we use the following conventions:
	\begin{align*}
	\Rm(X,Y,W,Z)&=\spann{\nabla_X\nabla_YZ-\nabla_Y\nabla_XZ-\nabla_{[X,Y]}Z,W},\\
	\Ric(V,W)&=\tr_\gamma \Rm(V,\cdot,W,\cdot),\\
	\operatorname{R}&=\tr_\gamma\Ric.
	\end{align*}
	We will use $\nabla$ for the Levi--Civita connection on $(\Sigma,\gamma)$ and $\nabla^k$ for the $k$-th tensor derivative. We use $\btr{\cdot}:=\btr{\cdot}_{\gamma}$ for all respective tensor norms induced by $\gamma$, where we will usually omit $\gamma$ for convenience when it is clear from the context. By slight abuse of notation, we will denote the gradient of a function $f$ on $(\Sigma,\gamma)$ by $\nabla f$. We denote the tracefree part of a $(0,2)$-tensor $T$ as $\accentset{\circ}{T}:=T-\tr_\gamma T\gamma$.
	
	We briefly recall the well-known Gauß and Codazzi Equations in the case of a codimension-$2$ surface in $\R^{1,3}$. For proofs, we refer to \cite[Section 2]{wolff1}.
	\begin{prop}[Gauß Equations]\label{prop_nullgauß}
		\begin{align*}
		\Rm_{ijkl}&=\frac{1}{2}\chi_{jl}\ul{\chi}_{ik}+\frac{1}{2}\ul{\chi}_{jl}\chi_{ik}-\frac{1}{2}\chi_{jk}\ul{\chi}_{il}-\frac{1}{2}\ul{\chi}_{jk}\chi_{il},\\
		\Ric_{ik}&=\frac{1}{2}\theta\ul{\chi}_{ik}+\frac{1}{2}\ul{\theta}\chi_{ik}-\frac{1}{2}(\chi\cdot\ul{\chi})_{ik}-\frac{1}{2}(\ul{\chi}\cdot\chi)_{ik},\\
		\operatorname{R}&=\mathcal{H}^2-\newbtr{\vec{\two}}^2.
		\end{align*}
	\end{prop}
	\begin{prop}[Codazzi Equations]\label{prop_nullcodazzi}
		\begin{align*}
		\nabla_i\ul{\chi}_{jk}-\nabla_j\ul{\chi}_{ik}
		&=-\zeta_j\ul{\chi}_{ik}+\zeta_i\ul{\chi}_{jk},\\
		\nabla_i\chi_{jk}-\nabla_j\chi_{ik}&=+\zeta_j\chi_{ik}-\zeta_i\chi_{jk}.
		\end{align*}
	\end{prop}	
	
	\subsection{Geometry on the standard Minkowski lightcone}\label{subsec_prelim_nullgeom}
		We briefly discuss the properties of the standard lightcone in the $3+1$-Minkwoski spacetime. For more details and helpful references, we refer the interested reader to \cite[Section 3]{wolff1}. To discuss the lightcone and its properties, it is most convenient to introduce the null coordinates $v\definedas r+t$, $u\definedas r-t$ on $\R^{1,3}$, as the Minkwoski metric $\eta$ can be written as
		\[
			\eta=\frac{1}{2}\left(\d u\d v+\d v\d u\right)+r^2\d\Omega^2
		\]
		with $r=r(u,v)=\frac{u+v}{2}$, and the future-pointing lightcone $\mathcal{N}=C(0)_+$ is given as the set $\{u=0\}$. Note that $\mathcal{N}$ has the induced degenerate metric 
		\[
			r^2\d\Omega^2,
		\]
		and is generated by the geodesic integral curves of $\ul{L}:=2\partial_v$. Note that $\ul{L}$ is future-pointing. Recall that the null generator $\ul{L}$ of a null hypersurface is both tangential and normal to $\mathcal{N}$, and by choice of $\ul{L}$ we have $\ul{L}(r)=1$. Thus, $r$ restricts to an affine parameter along $\mathcal{N}$. In particular, we can represent any spacelike cross section $\Sigma$ of $\mathcal{N}$ (which intersects any integral curve of $\ul{L}$ exactly once) as a graph over $\Sbb^2$, i.e., $\Sigma=\Sigma_\omega=\{\omega=r\}\subseteq\mathcal{N}$. In particular, $\Sigma$ has the induced metric
		\[
			\gamma=\omega^2\d\Omega^2,
		\]
		so $(\Sigma,\gamma)$ is conformally round. Conversely, for any conformally round metric $\gamma_\omega=\omega^2\d\Omega^2$ there exists a unique spacelike cross section $\Sigma_\omega$ such that $(\Sigma_\omega,\gamma_\omega)$ embeds into $\mathcal{N}$, where we will omit the subscript $\omega$ in the following when it is clear from the context. This observation is similar to an idea developed by Fefferman--Graham \cite{feffermangraham}, and their construction indeed yields the standard lightcone in the $1+3$-Minkowski spacetime in the case of the round $2$-sphere.
		
		We now represent the extrinsic curvature of a spacelike cross section $(\Sigma,\gamma)$ of $\mathcal{N}$ as a codimension-$2$ surface with respect to a particular null frame. Recall that the null generator $\ul{L}$ is both tangent and normal to $\mathcal{N}$, in particular $\ul{L}$ is normal to any spacelike cross section $(\Sigma,\gamma)$. We further consider a normal null vector field $L$ along $\Sigma$ such that $\eta(\ul{L},L)=2$. This uniquely determines $L$ such that $\{\ul{L},L\}$ form a frame of the normal bundle $T^\perp\Sigma$ of $\Sigma$. Note that $L$ is future-pointing.
		We then find the following, cf. \cite[Proposition 5]{wolff1}:
		\begin{prop}\label{prop_minkowskilightcone}
			For any spacelike cross section $(\Sigma,\gamma)$ in $\mathcal{N}$, we find
			\begin{enumerate}
				\item[\emph{(i)}] $\gamma=\omega^2\d\Omega^2$,\,\,\, $\ul{\chi}=\frac{1}{\omega}\gamma$,\,\,\, $\ul{\theta}=\frac{2}{\omega}$,
				\item[\emph{(ii)}] $\chi=\frac{1}{\omega}(1+\btr{\nabla\omega}^2)\gamma-2\Hess\, \omega$
				\item[\emph{(iii)}] $\theta=2\left(\frac{1}{\omega}+\frac{\btr{\nabla\omega}^2}{\omega}-\Delta \omega\right)$,
				\item[\emph{(iv)}] $\zeta=-\frac{\d\omega}{\omega}$,
			\end{enumerate}
			where $\Hess$ and $\Delta$ denote the Hessian and Laplace--Beltrami operator on $(\Sigma,\gamma)$, respectively.
		\end{prop}
		As stated in \cite[Remark 1]{wolff1}, the fact that $\accentset{\circ}{\ul{\chi}}=0$ for all spacelike cross sections yields that
		\begin{align}\label{eq_gaußcurvature}
		R=\frac{1}{2}\mathcal{H}^2
		\end{align}
		by the twice contracted Gauß equation Proposition \ref{prop_nullgauß}. We would like to emphasize that $\mathcal{H}^2$ refers by definition to the signed Lorentzian length of the mean curvature tensor and can therefore be (locally) negative despite the suggestive power of $2$ as an exponent. 
		
		Moreover, we find that
		\begin{align}
			\accentset{\circ}{\vec{\two}}=-\frac{1}{2}\accentset{\circ}{\chi}\ul{L},
		\end{align}
		so $\newbtr{\accentset{\circ}{\vec{\two}}}^2=0$ although $\accentset{\circ}{\vec{\two}}\not=0$ in general. Instead, we define a tensor-representative on $\Sigma$ that encodes the same extrinsic geometric information of the spacelike cross section. As in \cite{wolff1}, we define the \emph{scalar second fundamental form} $A$ as
		\begin{align}\label{eq_scalarA}
			A:=\ul{\theta}\chi,
		\end{align}
		and note that $\tr_\Sigma A=\mathcal{H}^2$. Then $A$ has the following properties, cf. \cite[Proposition 7]{wolff1}:
		\begin{prop}\label{prop_codazziminkowski2}
			Let $(\Sigma,\gamma)$ be a spacelike cross section of $\mathcal{N}$. Then $\nabla A$ is totally symmetric, i.e.,
			\begin{align}
			\nabla_iA_{jk}=\nabla_jA_{ik}.
			\end{align}
			In particular, we find
			\begin{align}\label{eq_Agradientestimate}
			\btr{\nabla A}^2\ge \frac{3}{4}\btr{\nabla\mathcal{H}^2}^2,
			\end{align}
			and $\accentset{\circ}{A}=0$ if and only if $\mathcal{H}^2$ is a strictly positive constant.
		\end{prop}
		We also refer to Al\'ias--C\'anovas--Rigoli \cite[Section 4]{aliascanovasrigoli} for an independent proof of the last equivalence in terms of the Weingarten endomorphism.
		\begin{bem}\label{bem_codazzi}
			Note that
			\begin{align}
				\accentset{\circ}{A}=-\frac{4}{\omega}\accentset{\circ}{\Hess}_\gamma\omega=4\omega\accentset{\circ}{\Hess}_{\Sbb^2}\left(\frac{1}{\omega}\right),
			\end{align}
			where ${\Hess}_{\Sbb^2}$ denotes the Hessian on $(\Sbb^2,\d\Omega^2)$. Using this, one can also verify by computation in coordinates that $\accentset{\circ}{\Hess}_{\Sbb^2}\left(\frac{1}{\omega}\right)=0$ if and only if $\omega$ is of the form
			\begin{align}\label{eq_metricconstantscalar}
				\omega(\vec{x})=\frac{\rho}{\sqrt{1+\norm{\vec{a}}^2}+\vec{a}\cdot\vec{x}}
			\end{align}
			for $\vec{x}\in\Sbb^2$ and a fixed vector $\vec{a}\in\R^3$, which are exactly the metrics of constant scalar curvature on $\Sbb^2$. It is a well-known fact that all such metrics can be obtained from the round metric by a suitable M\"obius transformation, cf. \cite[Proposition 6]{marssoria}, \cite[Section 4]{chenwang}. Moreover, the metrics \eqref{eq_metricconstantscalar} describe exactly the images of round spheres after a suitable Lorentz transformation in $\operatorname{SO}^+(1,3)$ in the Minkowski spacetime, which leave the lightcone $\mathcal{N}$ invariant. These observations illustrate once again the well-known fact that the M\"obius group is isomorphic to the restricted Lorentz group $\operatorname{SO}^+(1,3)$, which we will discuss in more detail in the next Subsection.
		\end{bem}
		Note that $A$ can also be interpreted as the null second fundamental form $\chi$ in a specific choice of gauge. In this gauge, the null generator $\widetilde{\ul{L}}$ along a spacelike cross section $\Sigma=\Sigma_\omega$ is given by
		\[
			\widetilde{\ul{L}}=\omega\partial_v,
		\]
		which is precisely $\frac{1}{2}$ times the position vector of points in $\Sigma$ in $C(0)_+$ with respect to the origin. Up to this factor of $2$, this particular choice of gauge and the conformally round structure of cross sections of the standard lightcone have been independently utilized by Le in a recent proof of an effective uniformization theorem, see \cite{le3}. See also \cite{le1, le2}.
	\subsection{The restricted Lorentz group}\label{subsec_prelim_lorentz}
	We briefly recall some well-known facts about the linear isometries of the Minkowski spacetime, see e.g. \cite{misnerthornewheeler}. The Lorentz group $\operatorname{O}(1,3)$ is the group of matrizes $L$ in $\R^{4\times 4}$ such that
	\[
	L^T\eta L=\eta
	\]
	for $\eta$ in Cartesian coordinates, i.e., the group of linear isometries of $(\R^{1,3},\eta)$. The full Lorentz group is a $6$-dimensional non-compact non-connected Lie group, but for the purpose of this work we will restrict our attention to the \emph{restricted Lorentz group} $\operatorname{SO}^+(1,3)$, which denotes the identity component of $\operatorname{O}(1,3)$. Note that $\operatorname{SO}^+(1,3)$ consists of all rotations and boosts that preserve the time-orientation (with respect to $\partial_t$), see Proposition \ref{prop_decomplorentz} below. In particular, $\Lambda(C(0)_\pm)=C(0)_\pm$ for all Lorentz transformations $\Lambda\in \operatorname{SO}^+(1,3)$.
	
	Here, we identify a rotation of the spatial coordinates $D$ in $\operatorname{SO}^+(1,3)$ with the matrix
	\[
	D=\begin{pmatrix}
	1&0\\
	0&R
	\end{pmatrix},
	\]
	where $R\in \operatorname{SO}(3)$ is a rotation in $\R^3$. For $\vec{a}\in\R^3\setminus{\{0\}}$, we define the rotation $D_{\vec{a}}$ to be the rotation that maps $\partial_3$ to $\frac{\vec{a}}{\btr{\vec{a}}}$ wihtout any additional rotation in the plane perpendicular to $\vec{a}$ in $\R^3$. Note that this uniquely determines $D_{\vec{a}}$.
		
	 A \emph{special Lorentz boost} $\Lambda_a$ is given by the matrix
	\[
	\Lambda_a=\begin{pmatrix}
	b&0&0&a\\
	0&1&0&0\\
	0&0&1&0\\
	a&0&0&b
	\end{pmatrix}
	\]
	for $a\in\mathbb{R}$ and $b:=\sqrt{1+a^2}$. More generally, we may consider a \emph{Lorentz boost} $\Lambda_{\vec{a}}$ in direction of $\vec{a}\in\R^3$. We may take the following decomposition as a definition of $\Lambda_{\vec{a}}$:
	\[
		\Lambda_{\vec{a}}:=
		\begin{cases}
			\operatorname{Id}_{\R^{1,3}}&\text{if }\vec{a}=0,\\
			D_{\vec{a}}\circ \Lambda_{\btr{\vec{a}}}\circ D^{-1}_{\vec{a}}&\text{if }\vec{a}\not=0.
		\end{cases}
	\]
	Note that
	\begin{align}\label{eq_lorentzboost}
	\Lambda_{\vec{a}}(\partial_t)=\left(\sqrt{1+\newbtr{\vec{a}}^2},\vec{a}\right).
	\end{align}
	
	For a general Lorentz transformation $\Lambda\in \operatorname{SO}^+(1,3)$, we recall the following well-known decomposition and include a brief proof for the sake of completeness.
	\begin{prop}\label{prop_decomplorentz}
		Let $\Lambda\in \operatorname{SO}^+(1,3)$. Then there exists $\vec{a}\in\R^3$ and a rotation $D$ defined as above such that
		\[
		\Lambda=\Lambda_{\vec{a}}\circ D.
		\]
	\end{prop}
	\begin{bem}\label{bem_decomplorentz}
		Note that $\vec{a}$ and $D$ are uniquely determined up to a choice of Cartesian coordinates, i.e., up to a choice of a positively oriented orthonormal basis on $\R^3$. As two such orthonormal frames are related by a uniquely determined rotation, $\vec{a}$ and $D$ transform under this rotation upon a change of the orthonormal frame. Hence, for fixed $\Lambda$ we may choose the orthonormal frame such that $D=\operatorname{Id}$.
	\end{bem}
	\begin{proof}
		Let $x^1$, $x^2$, $x^3$ denote Cartesian coordinates on $\R^3$. As $\Lambda\in \operatorname{SO}^+(1,3)$, there exists an $\vec{a}\in\R^3$ such that
		\[
		\Lambda(\partial_t)=\left(\sqrt{1+\btr{\vec{a}}^2},\vec{a}\right).
		\]
		By \eqref{eq_lorentzboost}, $L:=\Lambda^{-1}_{\vec{a}}\circ\Lambda$ is a linear isometry of the Minkowski spacetime with $L(\partial_t)=\partial_t$.
		Further $L(\partial_i)=(b_i,\vec{b}_i)$ for some $b_i\in\R$, $\vec{b}_i\in\R^3$, $i=1,2,3$. By Linearity,
		\[
		L(\partial_i-b_i\partial_t)=(0,\vec{b}_i),
		\]
		and as $L$ is an isometry, this implies that
		\[
		1-b_i^2=\newbtr{\vec{b_i}}^2=1+b_i^2.
		\]
		Thus $b_i=0$ for all $i=1,2,3$ and $L$ is of the form
		\[
		L=\begin{pmatrix}
		1&0\\
		0&R
		\end{pmatrix},
		\]
		for some $R\in\R^{3\times3}$. In particular, $R$ is a linear isometry on $\R^3$ and hence $L$ is either a reflection in the spacial directions or a rotation $D$ as defined above. As the restricted Lorentz group is the identity component of the Lorentz group, it does not contain reflections. This concludes the proof.
	\end{proof}
	
\section{An associated $4$-vector in the Minkowski spacetime}\label{sec_as4vector}
	Let $(\Sigma,\gamma)$ be a spacelike cross section of the future-pointing lightcone. Then, we define the \emph{associated $4$-vector ${\textbf{Z}}$ of $\Sigma$ in the ambient Minkowski spacetime} as
	\begin{align}\label{eq_associatedZ}
	\begin{split}
	{\textbf{Z}}(\Sigma)^0&:=\frac{1}{\btr{\Sigma}}\int_\Sigma t\d\mu_\Sigma,\\
	{\textbf{Z}}(\Sigma)^i&:=\frac{1}{\btr{\Sigma}}\int_\Sigma x^i\d\mu_\Sigma,
	\end{split}
	\end{align}
	where $x^i$ denotes the restriction of the standard Euclidean coordinates on $\R^3$ to $\Sigma$. In particular, the spacial coordinates are similarly defined as the Euclidean center of a surface in $\R^3$, see e.g. \cite{cederbaumsakovich,metzger}. In fact, up to rescaling to unit length and thus projecting to the hyperboloid $\{p\in\R^{1,3}\colon \eta(p,p)=-1\}$, this associated $4$-vector is directly equivalent to a notion of hyperbolic center (of mass) defined by Cederbaum--Cortier--Sakovich for surfaces in asymptotically hyperbolic initial data sets \cite{cederbaumcortiersakovich}. While they argue abstractly that this notion of center transforms equivariantly for surfaces on a hypberboloid in the Minkowski spacetime, the approach to prove this for spacelike cross sections of the standard lightcone presented here will be more explicit and computational in nature.
	
	By our previous considerations, any spacelike cross sections $(\Sigma,\gamma)$ is conformally round with conformal factor $\omega$ and $\Sigma=\graph_{\mathbb{S}^2}\omega$. Thus, for any spacelike cross section along the future-pointing lightcone, we have $t\vert_\Sigma=r\vert_\Sigma=\omega$, $x^i\vert_\Sigma=\omega f_i$, where $f_i$ denote first spherical harmonics. Hence,
	\begin{align}\label{eq_associatedZconformal}
	\begin{split}
	{\textbf{Z}}(\Sigma)^0&:=\frac{1}{\btr{\Sigma}}\int_{\mathbb{S}^2}\omega^3 \d\mu_{\mathbb{S}^2},\\
	{\textbf{Z}}(\Sigma)^i&:=\frac{1}{\btr{\Sigma}}\int_{\mathbb{S}^2} f_i\omega^3\d\mu_{\mathbb{S}^2}.
	\end{split}
	\end{align}
	
	In this context, we further note that the spatial coordinates closely resemble a notion of center of mass in conformal geometry, cf. \cite[Definition 2.8]{klainermanszeftel}, up to replacing $\omega^3$ by $\omega^2$ in the integrals. See also \cite{chang,changyang}. By a topological argument, it is a well-known fact that for any conformally round surface there exists a conformal transformation in the M\"obius group such that one can achieve
	\[
	\frac{1}{\btr{\Sigma}}\int_{\mathbb{S}^2} f_i\omega^2\d\mu_{\mathbb{S}^2}=0\text{ for all }i,
	\]
	see \cite{chang, changyang, onofri}. As we will show that the associated $4$-vector transforms equivariantly under Lorentz boosts in $\operatorname{SO}^+(1,3)$, see Proposition \ref{prop_Lorentztransform} below, we can achieve a similar balancing condition without relying on an implicit argument and uniquely identify the respective Lorentz transformation with respect to $\textbf{Z}$ up to isometries on $\Sbb^2$.
	
	As the future-pointing lightcone is mapped onto itself under any Lorentz transformation $\Lambda$ in $\operatorname{SO}^+(1,3)$, the image of any smooth spacelike cross section $\Sigma$ of the lightcone is itself a smooth, spacelike cross section $\Lambda(\Sigma)$. 
	\begin{lem}\label{lem_conformalfactorLorentztzransformation}
		Let $\Sigma$ be a spacelike cross section of the future-pointing standard Minkowski lightcone $\mathcal{N}$ with conformal factor $\omega$, and consider ${\Lambda\in \operatorname{SO}^+(1,3)}$. Then the conformal factor $\omega_\Lambda$ of $\Lambda(\Sigma)$ is given by
		\[
		\omega_\Lambda=\frac{\omega\circ\Phi_\Lambda}{\sqrt{1+\newbtr{\vec{a}}^2}-\vec{a}^if_i},
		\]
		where $\Phi_\Lambda\colon\mathbb{S}^2\to \mathbb{S}^2$ is a diffeomorphism, and $\Phi_\Lambda$ and $\vec{a}\in\R^3$ are uniquely determined by $\Lambda$.
	\end{lem}
	A straightforward computation then yields the following as a corollary:
	\begin{kor}\label{kor_conformalfactorLorentztzransformation}
		Let $\Sigma$ be a spacelike cross section of $\mathcal{N}$ with Gauss curvature $\mathcal{K}$, and consider $\Lambda\in \operatorname{SO}^+(1,3)$. Then the Gauss curvature $\mathcal{K}_\Lambda$ of $\Lambda(\Sigma)$ satisfies
		\[
		\mathcal{K}_\Lambda=\mathcal{K}\circ\Phi_\Lambda.
		\]
	\end{kor}
	\begin{bem}\label{bem_conformalfactorLorentztzransformation}\,
		\begin{enumerate}
			\item[(i)] 	In particular, $\Phi_\Lambda$ is in fact a M\"obius transformation on $\mathbb{S}^2$ uniquely determined for each $\Lambda\in \operatorname{SO}^+(1,3)$ by the isomorphism between the M\"obius group and the restricted Lorentz group, cf. Remark \ref{bem_codazzi}. In the context of M\"obius transformations and the conformal geometry of $\mathbb{S}^2$, both the invariance of the Gauss curvature and the precise transformation of the conformal factor under a M\"obius transformation as stated in Lemma \ref{lem_conformalfactorLorentztzransformation} are well-known facts, see e.g. \cite{klainermanszeftel} for a short summary and further references. Here, we give a different proof of the statement in Lemma \ref{lem_conformalfactorLorentztzransformation} as we approach it from an extrinsic viewpoint with the Lorentz transformations acting as isometries on the ambient Minkowski spacetime. See also \cite[Proposition 2.12]{chenwang} for an independet proof in the Minkowski spacetime.
			\item[(ii)] For $\Sigma_\rho$, i.e., $\omega=\rho>0$ for some positive constant, we recover the well-known formula for STCMC surfaces, cf. Remark \ref{bem_codazzi} Equation \eqref{eq_metricconstantscalar}. Direct computation gives
			\[
			{\textbf{Z}}(\Lambda(\Sigma_\rho))=\rho(\sqrt{1+\newbtr{\vec{a}}^2},\vec{a}).
			\]
			Hence, ${\textbf{Z}}(\Lambda(\Sigma_\rho))=\Lambda({\textbf{Z}}(\Sigma_\rho))$. 
			
			Moreover, for any timelike, future-pointing vector ${z}\in\R^{1,3}$, one can check by direct computation that the spacelike cross section corresponding to the conformal factor
			\begin{align}\label{eq_ZvectorSTCMC}
			\omega_{{z}}\colon=\frac{-\eta({z},{z})}{z^0-z^if_i}
			\end{align}
			is an STCMC surface with ${\textbf{Z}}(\Sigma_{\omega_{{z}}})={z}$. In this way, there is a one-to-one correspondence between timelike, future-pointing vectors and STCMC surfaces via ${\textbf{Z}}$.
		\end{enumerate}
	\end{bem}
	\begin{proof}[Proof of Lemma \ref{lem_conformalfactorLorentztzransformation}]
		Let us first consider a special Lorentz boosts $\Lambda_a$ as defined in Subsection \ref{subsec_prelim_lorentz}. In particular, the standard coordinates transform as
		\begin{align*}
		t'&=bt+az,\\
		x'&=x,\\
		y'&=y,\\
		z'&=at+bz,
		\end{align*}
		where $t',x',y',z'$ denote different Cartesian coordinates, and further consider spherical coordinates $t,r,\theta,\varphi$, and $t',r',\theta',\varphi'$ given via the set of transformations 
		\begin{align}
		\begin{split}\label{eq_transformation1}
		x^1=&r\sin\theta\sin\varphi,\\
		x^2=&r\sin\theta\cos\varphi,\\
		x^3=&r\cos\theta.
		\end{split}
		\end{align}
		As the lightcone is invariant under $\Lambda_a$, we know that $\omega_\Lambda=t'=r'$, and by the above identities it is easy to check that
		\[
		r=r'(b-a\cos\theta').
		\]
		Hence,
		\begin{align}\label{eq_omegaLambd1}
		\omega_\Lambda(\theta',\varphi')=\frac{\omega(\theta,\varphi)}{b-a\cos\theta'}.
		\end{align}
		Moreover, we see that
		\begin{align}
		\begin{split}\label{eq_transformation4}
		\sin\theta\sin\varphi&=\frac{\sin\theta'}{b-a\cos\theta'}\sin\varphi',\\
		\sin\theta\cos\varphi&=\frac{\sin\theta'}{b-a\cos\theta'}\cos\varphi',\\
		\cos\theta&=\frac{b\cos\theta'-a}{b-a\cos\theta'}.
		\end{split}
		\end{align}
		In particular, we obtain $\varphi=\varphi'$ and
		\begin{align}\label{eq_thetathetaprime1}
		\theta=\arccos\left(\frac{b\cos\theta'-a}{b-a\cos\theta'}\right)
		\end{align}
		A direct computation gives $\theta(\theta')\to 0$ as $\theta'\to 0$ and $\theta(\theta')\to \pi$ as $\theta'\to \pi$, and
		\begin{align}\label{eq_thetathetaprime2}
		\frac{\d\theta}{\d \theta'}=\frac{1}{b-a\cos\theta'}>0.
		\end{align}
		Thus, as $b>\btr{a}$, we note that
		\[
		\Phi_a\colon (0,\pi)\times(0,2\pi)\to (0,\pi)\times(0,2\pi)\colon (\theta',\varphi')\mapsto \left(\arccos\left(\frac{b\cos\theta'-a}{b-a\cos\theta'}\right),\varphi'\right)
		\]
		extends smoothly to a diffeomorphism $\Phi_a\colon \Sbb^2\to\Sbb^2$. This establishes the claim in the case of a special Lorentz boost.
		
		For a rotation $D$, let $\Phi_D\colon \mathbb{S}^2\to\mathbb{S}^2, \vec{x}\mapsto R(\vec{x})$ denote the corresponding diffeomorphism on $\Sbb^2$. In fact, $\Phi_D$ acts as an isometry on the round sphere. It is easy to check that
		\[
		\omega_D(\vec{x})=\omega\circ\Phi^{-1}_D(\vec{x}).
		\]
		In view of Proposition \ref{prop_decomplorentz} and choosing an appropriate ON-frame as in Remark \ref{bem_decomplorentz}, the claim directly follows by observing that for $D=D_{\vec{a}}$
		\begin{align*}
		\left(\sqrt{1+\newbtr{\vec{a}}^2}-\newbtr{\vec{a}}\cos\theta\right)(\Phi^{-1}_D)(\vec{x})
		&=\sqrt{1+\newbtr{\vec{a}}^2}-\left(\newbtr{\vec{a}}\partial_3\right)\cdot \Phi^{-1}_D(\vec{x})\\
		&=\sqrt{1+\newbtr{\vec{a}}^2}-\Phi_D\left(\newbtr{\vec{a}}\partial_3\right)\cdot \vec{x}\\
		&=\sqrt{1+\newbtr{\vec{a}}^2}-\vec{a}\cdot\vec{x}\\
		&=\sqrt{1+\newbtr{\vec{a}}^2}-\vec{a}^if_i,
		\end{align*}
		where we used that $\Phi_D$ is an isometry on the round sphere. Note further that by the above decomposition
		\[
		\Phi_\Lambda=\Phi_D\circ\Phi_{\newbtr{\vec{a}}}\circ \Phi_D^{-1}.
		\]
	\end{proof}
	From this, the following proposition is easily established:
	\begin{prop}\label{prop_Lorentztransform}
		Let $\Sigma$ be a spacelike cross section with associated $4$-vector ${\textbf{Z}}$, and consider $\Lambda\in \operatorname{SO}^+(1,3)$. Then $\btr{\Lambda(\Sigma)}=\btr{\Sigma}$, and ${\textbf{Z}}$ is a timelike, future-pointing vector, such that
		\[
		\Lambda({\textbf{Z}}(\Sigma))={\textbf{Z}}(\Lambda(\Sigma)).
		\]
	\end{prop}
	\begin{bem}\label{bem_Lorentztransform}
		In fact, it holds that
		\[
		\int_{\Lambda(\Sigma)}f\circ\Phi_\Lambda=\int_{\Sigma}f
		\]
		for any continuous function $f$ on $\Sbb^2$. Hence, by Corollary \ref{kor_conformalfactorLorentztzransformation}, we find
		\begin{align}\label{eq_gaussL2Lorentztransform}
		\norm{\mathcal{K}_\Lambda-1}^2_{L^2(\Lambda(\Sigma)}=\norm{\mathcal{K}-1}^2_{L^2(\Sigma)}.
		\end{align}
	\end{bem}
	\begin{proof}[Proof of Proposition \ref{prop_Lorentztransform}]
		In view of \eqref{eq_associatedZconformal} it is straightforward to check that $\textbf{Z}$ transforms accordingly under spatial rotations as they act as isometries on $\Sbb^2$. In view of Proposition \ref{prop_decomplorentz} it remains to prove the claim for special Lorentz boosts $\Lambda_a$.
		
		Now let $\Lambda_a$ be a special Lorentz boost for $a\in\R$. Recall from the proof of Lemma \ref{lem_conformalfactorLorentztzransformation} that
		\[
		\omega_\Lambda(\theta',\varphi')=\frac{\omega(\theta(\theta'),\varphi')}{b-a\cos\theta'},
		\]
		with $\theta$ given by \eqref{eq_thetathetaprime1}. From this, we can explicitly compute that
		\begin{align*}
		\sin\theta&=\frac{\sin\theta'}{b-a\cos\theta'},\\
		b+a\cos\theta&=\frac{1}{b-a\cos\theta'},\\
		b\cos\theta+a&=\frac{\cos\theta'}{b-a\cos\theta'}.
		\end{align*}
		Then, using \eqref{eq_thetathetaprime2} and substitution gives
		\begin{align*}
		\btr{\Lambda(\Sigma)}
		&=\int\limits_{0}^{2\pi}\int\limits_{0}^\pi\frac{\omega(\theta(\theta'),\varphi(\varphi'))^2}{(b-a\cos\theta')^2}\sin\theta'\d\theta'\d\varphi'
		=\int\limits_{0}^{2\pi}\int\limits_{0}^\pi\omega(\theta(\theta'),\varphi)^2\sin\theta(\theta')\frac{\d\theta}{\d\theta'}\d\theta'\d\varphi
		=\btr{\Sigma},
		\end{align*}
		and
		\begin{align*}
		\int_{\mathbb{S}^2}\omega_\Lambda^3\d\mu=&\,\int\limits_{0}^{2\pi}\int\limits_{0}^\pi\frac{\omega(\theta(\theta'),\varphi(\varphi'))^3}{(b-a\cos\theta')^3}\sin\theta'\d\theta'\d\varphi'\\
		=&\,b\int\limits_{0}^{2\pi}\int\limits_{0}^\pi{\omega(\theta(\theta'),\varphi(\varphi'))^3}\sin\theta\frac{\d\theta}{\d\theta'}\d\theta'\d\varphi'+a\int\limits_{0}^{2\pi}\int\limits_{0}^\pi{\omega(\theta(\theta'),\varphi(\varphi'))^3}\cos\theta\sin\theta\frac{\d\theta}{\d\theta'}\d\theta'\d\varphi'\\
		=&\,b\int\limits_{0}^{2\pi}\int\limits_{0}^\pi{\omega(\theta,\varphi)^3}\sin\theta\d\theta\d\varphi+a\int\limits_{0}^{2\pi}\int\limits_{0}^\pi{\omega(\theta,\varphi)^3}\cos\theta\sin\theta\d\theta\d\varphi\\
		=&\,b\int_{\Sbb^2}\omega^3\d\mu+a\int_{\Sbb^2}\omega^3f_3\d\mu.
		\end{align*}
		Similarly
		\begin{align*}
		\int_{\Sbb^2}\omega_\Lambda^3f_1\d\mu&=\int_{\Sbb^2}\omega^3f_1\d\mu,\\
		\int_{\Sbb^2}{\omega}_\Lambda^3f_2\d\mu&=\int_{\Sbb^2}\omega^3f_2\d\mu,\\
		\int_{\Sbb^2}{\omega}_\Lambda^3f_3\d\mu&=b\int_{\Sbb^2}\omega^3f_3\d\mu+a\int_{\Sbb^2}\omega^3\d\mu.
		\end{align*}
		By \eqref{eq_associatedZconformal} it follows that
		\[
		{\textbf{Z}}(\Lambda_a(\Sigma))=\Lambda_a({\textbf{Z}}(\Sigma)),
		\]
		and it remains to show that $\textbf{Z}(\Sigma)$ is timelike, future-pointing. 
		By the above, we may assume without loss of generality that
		\[
		{\textbf{Z}}=\begin{pmatrix}
		{\textbf{Z}}^0\\
		0\\
		0\\
		{\textbf{Z}}^3
		\end{pmatrix},
		\]
		after a suitable rotation of the spatial coordinates in the ambient Minkowski spacetime. Thus
		\[
		\btr{\Sigma}\left({\textbf{Z}}^0-\newbtr{{\textbf{Z}}^3}\right)\ge \int_{\Sbb^2}\omega^3(1-\btr{f_3})>0,
		\]
		which implies that $\textbf{Z}$ is timelike, future-pointing.
	\end{proof}
	Using the one-to-one correspondence between STCMC surfaces and timelike, future-pointing vectors induced by the definition of ${\textbf{Z}}$, we define an a-priori class suitable for our purposes.
	\begin{defi}\label{defi_kappabounded}
		Let $\kappa>0$. We say a spacelike cross section $\Sigma$ is \emph{$\kappa$-bounded}, if
		\[
		(1+\kappa)^{-1}\omega_{{\textbf{Z}}}\le\omega\le (1+\kappa)\omega_{{\textbf{Z}}},
		\]
		where ${\textbf{Z}}={\textbf{Z}}(\Sigma)$.
	\end{defi}
	\begin{bem}\label{kappa_bounded}
		As all surfaces under consideration are smooth any spacelike cross section is of course $\kappa$-bounded for some appropriate $\kappa=\kappa(\omega)>0$. As all the estimates derived below assume $\kappa$ to be fixed a-priori, we may rephrase them as depending an a suitable sup-bound on $\omega$ with constants explicitly depending on this sup-bound.
	\end{bem}
	We close this subsection with some lemmas to be used later.
	\begin{lem}\label{lem_kappabounded}
		$\kappa$-boundedness is preserved under rescaling and Lorentz transformations \linebreak ${\Lambda\in \operatorname{SO}^+(1,3)}$.
	\end{lem}
	\begin{proof}
		As ${\textbf{Z}}(\Sigma_{\omega_{{\textbf{Z}}}})={\textbf{Z}}$ we can conclude that
		\[
		\Lambda(\Sigma_{\omega_{{\textbf{Z}}}})=\Sigma_{\omega_{\Lambda({\textbf{Z}})}}
		\]
		using Proposition \ref{prop_Lorentztransform} and the one-to-one correspondence between STCMC surfaces on the future-pointing standard lightcone and timelike, future-pointing vectors as in Remark \ref{bem_conformalfactorLorentztzransformation} (ii). The claim then follows directly from Proposition \ref{prop_Lorentztransform} and the scaling properties of $\textbf{Z}$.
	\end{proof}
	\begin{lem}\label{lem_kappabounded2}
		Let $\Sigma$ be $\kappa$-bounded. Then up to isometries on $\Sbb^2$ there exists a unique Lorentztransformation $\Lambda\in \operatorname{SO}^+(1,3)$ such that $\omega_\Lambda$ satisfies 
		\[
		\int_{\Sbb^2}f_i\omega_\Lambda^3\d\mu=0\text{ for all }i=1,2,3\,
		\]
		and
		\[
		(1+\kappa)^{-2}r\le \omega_\Lambda\le (1+\kappa)^2r,
		\]
		where $r=\sqrt{\frac{\btr{\Sigma}}{4\pi}}$ is the area radius of $\Sigma$.
	\end{lem}
	\begin{proof}
		Integrating the inequality in Definition \ref{defi_kappabounded} immediately yields that
		\begin{align}\label{eq_lemmakappabounded1}
		(1+\kappa)^{-1}{r_{\textbf{Z}}}\le r\le (1+\kappa){r_{\textbf{Z}}},
		\end{align}
		where $r_{{\textbf{Z}}}:=\sqrt{-\eta({\textbf{Z}}(\Sigma),{\textbf{Z}}(\Sigma))}$.
		Let $\Lambda\in \operatorname{SO}^+(1,3)$ be the Lorentz boost (unique up to rotations) such that 
		\[
		\Lambda({\textbf{Z}}(\Sigma))=\left(r_{{\textbf{Z}}},0,0,0\right)
		\] 
		By Proposition \ref{prop_Lorentztransform} it follows that
		\[
		\int_{\Sbb^2}f_i\omega_\Lambda^3\d\mu=0.
		\]
		In particular, $\omega_{\Lambda(\textbf{Z})}=r_{\textbf{Z}}$. Hence
		\[
		(1+\kappa)^{-1}r_{\textbf{Z}}\le \omega_\Lambda\le (1+\kappa)r_{\textbf{Z}},
		\]
		as $\kappa$-boundedness is preserved by Lemma \ref{lem_kappabounded}. The claim then follows by Equation \eqref{eq_lemmakappabounded1}.
	\end{proof}
	\begin{bem}\label{bem_kappabounded}
		In fact, it holds that
		\[
		(1+\kappa)^{-1}r\le \omega_\Lambda\le (1+\kappa)^2r
		\]
		as $r_{\textbf{Z}}:=\sqrt{-\eta(\textbf{Z},\textbf{Z})}\ge r$ with equality if and only if $\omega=\omega_{\textbf{Z}}$. To see this, we can assume wlog that $\textbf{Z}=\left(r_{\textbf{Z}},\vec{0}\right)$ using Proposition \ref{prop_Lorentztransform} as above. Thus
		\[
		\btr{\Sigma_{\omega_{\textbf{Z}}}}=\frac{1}{\btr{\Sigma}^2}\int_{\mathbb{S}^2}\left(\int_{\mathbb{S}^2}\omega^3\right)^2=4\pi\frac{\left(\int_{\mathbb{S}^2}\omega^3\right)^2}{\left(\int_{\mathbb{S}^2}\omega^2\right)^2}\ge \btr{\Sigma},
		\]
		which follows from applying the H\"older inequality twice.
	\end{bem}
\section{A scaling invariant estimate on the lightcone}\label{sec_estimate}
	In this section, we want to prove a scaling invariant, geometric estimate of the form
	\begin{align}\label{eq_tracefreeA}
	\btr{\Sigma}\int_\Sigma \btr{A-\frac{\fint\mathcal{H}^2}{2}\gamma}^2\le C\btr{\Sigma} \int_\Sigma\newbtr{\accentset{\circ}{A}}^2.
	\end{align}
	for any spacelike cross section $\Sigma$ of the standard Minkowski lightcone, where $C>0$ is a uniform constant. We present two different proofs of such an estimate both requiring that the spacelike cross section under consideration satisfies $\mathcal{H}^2\ge 0$.
\subsection{A proof using null mean curvature flow}\label{subsec_nullmeancurvatureflow}
	We give a proof of the desired estimate \eqref{eq_tracefreeA} on the standard Minkowski lightcone using null mean curvature flow. Recall that null mean curvature flow along a null hypersurface with null generator $\ul{L}$ as first studied by Roesch--Scheuer \cite{roeschscheuer} is defined as the evolution of a family of spacelike cross sections $\Sigma_t$ with
	\[
		\frac{\d}{\d t}x=\frac{1}{2}\mathfrak{g}\left(\vec{\mathcal{H}},L\right),
	\]
	where $\vec{\mathcal{H}}$ is the mean curvature vector of $\Sigma_t$ in the ambient spacetime $(\mathfrak{M},\mathfrak{g})$ and $L$ is the unique null vector field normal to $\Sigma_t$ such that $\mathfrak{g}(\ul{L},L)=2$. In the case of the standard Minkowski lightcone, null mean curvature flow has recently been studied by the author in \cite{wolff1}. In particular, with respect to the null generator $\ul{L}=2\partial_v$, null mean curvature flow on the standard Minkowski lightcone is equivalent to the the scalar evolution equation
	\begin{align}
		\frac{\d}{\d t}\omega=-\frac{1}{2}\theta
	\end{align}
	for the conformal factor $\omega$. We note that this evolution equation, and hence null mean curvature flow along the standard Minkowski lightcone, is equivalent to $2d$-Ricci flow in the conformal class of the round sphere, cf. \cite[Section 4]{wolff1}. Hence in this case, the fact that all singularity models for null mean curvature flow are STCMC surfaces (\cite[Corollary 11]{wolff1}) follows directly from a classical result first proven by Hamilton \cite{hamilton1}, see also \cite{chow1, bartzstruweye, struwe, andrewsbryan}. Additionally, the author gave an independent proof of this result for spacelike cross sections with $\mathcal{H}^2>0$, cf. \cite[Theorem 12]{wolff1}. Due to the Gauss equation \eqref{eq_gaußcurvature} and the equivalence between the flows, this gives a new proof of Hamilton's classical result under its initial restriction of positive scalar curvature in the conformally round case, cf. \cite{hamilton1}.
	\begin{thm}\label{thm_tracefreeA1}
		Let $\{\Sigma_t\}_t$ be a family of spacelike cross sections evolving under null mean curvature flow with $\mathcal{H}^2\ge0$. Then
		\[
		\frac{\d }{\d t}\left(
		\btr{\Sigma}\int\frac{1}{2}(\mathcal{H}^2)^2-2\newbtr{\accentset{\circ}{A}}^2\right)\ge 0.
		\]
		As a consequence, we have
		\[
		\btr{\Sigma}\int\frac{1}{2}(\mathcal{H}^2)^2-2\newbtr{\accentset{\circ}{A}}^2\le 128\pi^2=\frac{1}{2}\left(\int\mathcal{H}^2\right)^2,
		\]
		with equality if and only if $\{\Sigma_t\}$ is a family of shrinking STCMC surfaces.
	\end{thm}
	\begin{bem}\label{bem_tracefreeA1}\,
		\begin{enumerate}
			\item[(i)] Using the evolution equation for $\mathcal{H}^2$ as computed in \cite[Lemma 8]{wolff1}, we have $\mathcal{H}^2>0$ for any $t>0$ by the strong maximum principle if $\mathcal{H}^2\ge 0$ for $\Sigma_0$, as we can rule out $\mathcal{H}^2\equiv0$ by Gauss--Bonnet. Thus, we may entirely rely on the special case proven independently by Hamilton \cite{hamilton1}, and the author \cite{wolff1}. However, as we may also rely on the more general result \cite{chow1, bartzstruweye, struwe, andrewsbryan}, we merely require the assumption $\mathcal{H}^2\ge 0$ for the application of Lemma \ref{lem_hoelder}, see below.
			\item[(ii)] The strategy presented here is motivated by a short, unpublished prove by Huisken of the inequality
			\[
			\norm{h-\frac{\fint H}{2}\gamma}_{L^2(\Sigma)}\le 2\newnorm{\accentset{\circ}{h}}_{L^2(\Sigma)}
			\]
			for strictly starshaped surfaces $\Sigma$ in $\R^3$ with $H>0$ using a result by Gerhardt \cite{gerhardt} for \emph{inverse mean curvature flow}, where $2$ is indeed the optimal constant in this case.
			\item[(ii)] Along a null hypersurface, inverse mean curvature flow is defined as the projection of codimension-$2$ inverse mean curvature flow in an ambient spacetime $(\mathfrak{M},\mathfrak{g})$ to the null hypersurface, i.e.,
			\[
			\frac{\d}{\d t}x=-\frac{1}{2}\frac{\mathfrak{g}(\vec{\mathcal{H}},L)}{\mathcal{H}^2}\ul{L}=\frac{1}{2\ul{\theta}}\ul{L},
			\]
			which is well-defined as long as $\ul{\theta}\not=0$ along the null hypersurface. As $\ul{\theta}=\frac{2}{\omega}$ for any spacelike cross section $\Sigma_\omega$ of the standard Minkowski lightcone by Proposition \ref{prop_minkowskilightcone}, inverse mean curvature flow is given as an ordinary differential equation rather than a parabolic system. More explicitly, for any spacelike cross section $\Sigma_\omega$ along the standard lightcone, the solution of inverse mean curvature flow is given by the smooth family $\{\Sigma_{\omega(s)}\}$ with 
			\[
			\omega(s)=\omega e^{\frac{s}{2}}.
			\]
			In particular, any scaling invariant quantity remains unchanged under the flow, and it is thus not suitable to prove Theorem \ref{thm_tracefreeA1}.
		\end{enumerate}
	\end{bem}
	Before proving Theorem \ref{thm_tracefreeA1}, we establish the following auxiliary lemma following from the H\"older inequality:
	\begin{lem}\label{lem_hoelder}
		Let $(X,\mu)$ be a finite measure space, and $f$ a bounded, non-negative, measurable function with $\int f\d\mu>0$. Then
		\[
		\int f\int f^2\le \mu(X)\int f^3
		\]
		with equality if and only if $f$ is constant.
	\end{lem}
	\begin{proof}
		Using the H\"older inequality twice, we get
		\begin{align*}
		\int f\int f^2\le \frac{\mu(X)}{\int f}\left(\int f^2\right)^2\le \frac{\mu(X)}{\int f}\int f\int f^3.
		\end{align*}
	\end{proof}
	\begin{proof}[Proof of Theorem \ref{thm_tracefreeA1}]
		Using the evolution equations as computed in \cite[Lemma 8]{wolff1}, we find
		\begin{align}
		\begin{split}\label{eq_evolutioneqations1}
		\frac{\,\operatorname{d}}{\,\operatorname{d}t}\mathcal{H}^2&=\Delta\mathcal{H}^2+\frac{1}{2}\left(\mathcal{H}^2\right)^2,\\
		\frac{\d}{\d t}\newbtr{\accentset{\circ}{A}}^2&=\Delta\newbtr{\accentset{\circ}{A}}^2-2\left(\btr{\nabla A}^2-\frac{1}{2}\btr{\nabla\mathcal{H}^2}^2\right)\le \Delta\newbtr{\accentset{\circ}{A}}^2-\frac{1}{2}\btr{\nabla \mathcal{H}^2}^2,
		\end{split}
		\end{align}
		using Proposition \ref{prop_codazziminkowski2}, Equation \eqref{eq_Agradientestimate}.
		Moreover, as $\eta(\vec{\mathcal{H}},L)=-\frac{1}{2}\theta$ by \eqref{eq_secondffnulldecomp} we find
		\[
		\frac{\d}{\d t}\d\mu_\gamma=-\frac{1}{2}\mathcal{H}^2\d\mu_\gamma
		\]
		by the Raychaudhuri equations, cf. \cite[Lemma 4 (i)]{wolff1}. A direct computation yields
		\begin{align*}
		\frac{\d }{\d t}\left(
		\btr{\Sigma}\int\frac{1}{2}(\mathcal{H}^2)^2-2\newbtr{\accentset{\circ}{A}}^2\right)
		\ge& \int\mathcal{H}^2\int\newbtr{\accentset{\circ}{A}}^2-\frac{1}{4}\int\mathcal{H}^2\int(\mathcal{H}^2)^2\\
		&+\,\btr{\Sigma}\int \frac{1}{2}\Delta(\mathcal{H}^2)^2-\btr{\nabla\mathcal{H}^2}^2+\frac{1}{2}(\mathcal{H}^2)^3-2\Delta \newbtr{\accentset{\circ}{A}}^2+\btr{\nabla \mathcal{H}^2}^2\\
		&\,-\btr{\Sigma}\int\frac{1}{2}\mathcal{H}^2\left(\frac{1}{2}(\mathcal{H}^2)^3-2\newbtr{\accentset{\circ}{A}}^2\right)\\
		=&\int\mathcal{H}^2\int \newbtr{\accentset{\circ}{A}}^2
		+\btr{\Sigma}\int\mathcal{H}^2\newbtr{\accentset{\circ}{A}}^2\\
		&+\frac{1}{2}\left(\btr{\Sigma}\int(\mathcal{H}^2)^3-\int\mathcal{H}^2\int(\mathcal{H}^2)^2\right)
		\ge 0,
		\end{align*}
		where we used Lemma \ref{lem_hoelder} in the last line, as all assumptions are satisfied by Gauss--Bonnet. Thus, we have proven the monotonicity.
		
		Note that $\btr{\Sigma}\int\left(\frac{1}{2}(\mathcal{H}^2)^2-2\newbtr{\accentset{\circ}{A}}^2\right)$ is scaling invariant, thus the monotonicity is also satisfied for any conformally equivalent flow, in particular for the flow rescaled to keep the area constant. Then the convergence to a round limit with $\mathcal{H}^2=\operatorname{const.}$ and $\newbtr{\accentset{\circ}{A}}^2=0$ follows by \cite[Theorem 12]{wolff1}. Hence,
		\[
		\btr{\Sigma}\int\frac{1}{2}(\mathcal{H}^2)^2-2\newbtr{\accentset{\circ}{A}}^2\le 128\pi^2=\frac{1}{2}\left(\int\mathcal{H}^2\right)^2,
		\]
		where the second identity holds by Gauss--Bonnet. Lastly, equality is achieved if and only if 
		\[
		\frac{\d }{\d t}\left(
		\btr{\Sigma}\int\frac{1}{2}(\mathcal{H}^2)^2-2\newbtr{\accentset{\circ}{A}}^2\right)=0,
		\]
		so by Proposition \ref{prop_codazziminkowski2} and Lemma \ref{lem_hoelder} if $\{\Sigma_t\}$ is a family of shrinking STCMC surfaces.
	\end{proof}
	From this, we obtain the desired estimate:
	\begin{thm}\label{thm_tracefreeA2}
		Let $(\Sigma,\gamma)$ be a spacelike cross section of the standard lightcone in the $1+3$-Minkowski spacetime with $\mathcal{H}^2\ge 0$. Then, we have that
		\[
		\btr{\Sigma}\int_\Sigma \btr{A-\frac{\fint\mathcal{H}^2}{2}\gamma}^2\le 3\btr{\Sigma} \int_\Sigma\newbtr{\accentset{\circ}{A}}^2,
		\]
		and equality holds if and only if $\Sigma$ is a surface of constant spacetime mean curvature.
	\end{thm}\pagebreak
	\begin{bem}\label{bem_tracefreeA2}\,
		\begin{enumerate}
			\item[(i)] Using the Gauss equation \eqref{eq_gaußcurvature}, we see that in fact
			\[
			\frac{1}{2}\fint\mathcal{H}^2=\frac{2}{r^2}
			\]
			by Gauss--Bonnet, where $r:=\sqrt{\frac{\btr{\Sigma}}{4\pi}}$ is the area radius of $\Sigma$. We have thus proven equivalently that
			\begin{align}\label{eq_tracefreeA2}
			\btr{\Sigma}\int_\Sigma \btr{A-\frac{2}{r^2}\gamma}^2\le 3\btr{\Sigma} \int_\Sigma\newbtr{\accentset{\circ}{A}}^2.
			\end{align}
			\item[(ii)] The constant $C=3$ is not optimal. See Theorem \ref{thm_tracefreeA3} below.
		\end{enumerate}
	\end{bem}
	\begin{proof}[Proof of Theorem \ref{thm_tracefreeA2}]
		We first rewrite the left-hand side as
		\begin{align*}
		\btr{\Sigma}\int_\Sigma \btr{A-\frac{\fint\mathcal{H}^2}{2}\gamma}^2
		&=\btr{\Sigma}\int_\Sigma\left(\btr{A}^2-\left(\fint\mathcal{H}^2\right)\mathcal{H}^2+\frac{1}{2}\left(\fint\mathcal{H}^2\right)^2\right)\\
		&=\btr{\Sigma}\int_\Sigma\btr{A}^2-\frac{1}{2}\left(\int_\Sigma\mathcal{H}^2\right)^2\\
		&=\btr{\Sigma}\int_\Sigma\btr{A}^2-128\pi^2
		\end{align*}
		By the strong maximum principle $\mathcal{H}^2\ge 0$ on $\Sigma$ is preserved under null mean curvature flow with initial data $\Sigma_0=\Sigma$. Hence, the claim follows from Theorem~\ref{thm_tracefreeA1}.
	\end{proof}
	Equivalenty, we have proven the following Corollary.
	\begin{kor}\label{kor_tracefreeA2}
		If $\mathcal{H}^2\ge 0$,
		\[
		\btr{\Sigma}\norm{\mathcal{H}^2-\fint\mathcal{H}^2}^2_{L^2(\Sigma)}\le 4\btr{\Sigma}\newnorm{\accentset{\circ}{A}}_{L^2(\Sigma)}^2.
		\]
	\end{kor}
	We close this subsection by a short proof of an improved gradient estimate for solutions of null mean curvature flow with $\mathcal{H}^2>0$, cf. \cite[Theorem 16]{wolff1}.
	\begin{thm}\label{thm_gradientestiamte}
		Let $\{\Sigma_t\}_{t\in[0,T_{max})}$ be a family of spacelike cross sections of the standard Minkowski lightcone with strictly positive spacetime mean curvature $\mathcal{H}^2>0$ evolving under null mean curvature flow. For any $p>\frac{1}{2}$, $\eta>0$, there exists $C_\eta>0$ only depending on $\eta$, $p$ and $\Sigma_0$, such that
		\begin{align*}
		\left\vert \nabla \mathcal{H}^2 \right\vert\le \eta^2(\mathcal{H}^2)^{p}+C_\eta.
		\end{align*}
		for all $t\in[0,T_{max})$.
	\end{thm}
	\begin{proof}
		As $\btr{\nabla\mathcal{H}^2}^2=\gamma^{ij}\partial_i\mathcal{H}^2\partial_j\mathcal{H}^2$, a straightforward computation using \eqref{eq_evolutioneqations1} gives
		\begin{align*}
		\frac{\d}{\d t}\btr{\nabla\mathcal{H}^2}^2=2\spann{\nabla\frac{\d}{\d t}\mathcal{H}^2,\nabla\mathcal{H}^2}+\frac{1}{2}\mathcal{H}^2\btr{\nabla\mathcal{H}^2}^2
		=2\spann{\nabla\Delta\mathcal{H}^2,\nabla\mathcal{H}^2}+\frac{5}{2}\mathcal{H}^2.
		\end{align*}
		Using the Bochner formula
		\begin{align}\label{eq_bochner}
		\Delta\btr{\nabla f}^2=2\spann{\nabla\Delta f,\nabla f}+2\btr{\Hess f}^2+2\Ric_\gamma(\nabla f,\nabla f)
		\end{align}
		for any smooth function $f$, and the fact that $\Ric_\gamma=\frac{1}{4}\mathcal{H}^2\gamma$ by the Gauss formula \eqref{eq_gaußcurvature}, we conclude that
		\[
		\frac{\d }{\d t}\btr{\nabla\mathcal{H}^2}^2=\Delta\btr{\nabla\mathcal{H}^2}^2-2\btr{\Hess\,\mathcal{H}^2}^2+2\mathcal{H}^2\btr{\nabla\mathcal{H}^2}^2.
		\]
		We note that this is consistent with the evolution of the gradient of the scalar curvature under $2d$-Ricci flow, as we would expect. Then, direct computation gives
		\begin{align*}
		\frac{\d }{\d t}\frac{\btr{\nabla\mathcal{H}^2}^2}{\mathcal{H}^2}
		&=\Delta\frac{\btr{\nabla\mathcal{H}^2}^2}{\mathcal{H}^2}+2\frac{\spann{\nabla \btr{\nabla\mathcal{H}^2}^2,\nabla \mathcal{H}^2}}{(\mathcal{H}^2)^2}-2\frac{\btr{\nabla\mathcal{H}^2}^4}{(\mathcal{H}^2)^3}-2\frac{\btr{\Hess\,\mathcal{H}^2}^2}{\mathcal{H}^2}+\frac{3}{2}\btr{\nabla\mathcal{H}^2}^2\\
		&\le \Delta\frac{\btr{\nabla\mathcal{H}^2}^2}{\mathcal{H}^2} +\frac{3}{2}\btr{\nabla\mathcal{H}^2}^2
		\end{align*}
		using Cauchy--Schwarz and Young's inequality. By \eqref{eq_evolutioneqations1}
		\[
		\left(\frac{\d}{\d t}-\Delta\right)\left(\frac{\btr{\nabla\mathcal{H}^2}^2}{\mathcal{H}^2}+3\newbtr{\accentset{\circ}{A}}^2\right)\le 0
		\]
		and therefore any initial bound is preserved for the sum. Thus
		\[
		\frac{\btr{\nabla\mathcal{H}^2}^2}{\mathcal{H}^2}\le \frac{\btr{\nabla\mathcal{H}^2}^2}{\mathcal{H}^2}+3\newbtr{\accentset{\circ}{A}}^2\le C(\Sigma_0).
		\]
		The desired estimate follows from multiplying by $\mathcal{H}^2$ and using Young's inequality.
	\end{proof}
	We recall that Chow \cite{chow1} directly extended the result by Hamilton \cite{hamilton1} in the conformally round case by showing that the scalar curvature has to become strictly positive under evolution of volume preserving $2d$-Ricci flow for sufficiently large times. As this flow is directly related to Ricci flow upon a suitable rescaling, any solution to $2d$-Ricci flow on $\mathbb{S}^2$ has strictly positive scalar curvature sufficiently close to the singular time. We thus obtain the following Corollary by the equivalence of the two flows in our setting:
	\begin{kor}
		Let $\{\Sigma_t\}_{t\in[0,T_{max})}$ be a family of closed, topological $2$-spheres evolving under Ricci flow. Then for all times $t$ sufficiently close to $T_{max}$, we have the following: For any $p>\frac{1}{2}$, $\eta>0$, there exists $C_\eta>0$ only depending on $\eta$, $p$ and $\Sigma_0$, such that
		\begin{align*}
			\left\vert \nabla \operatorname{R} \right\vert\le \eta^2\operatorname{R}^{p}+C_\eta.
		\end{align*}
	\end{kor}
\subsection{A proof using the Bochner formula}\label{subsec_bochner}
	We now give a different proof for the desired estimate \eqref{eq_tracefreeA}, by improving Corollary \ref{kor_tracefreeA2}. This will in fact yield the optimal estimate.
	\begin{prop}\label{thm_tracefreeA3}
		If $\mathcal{H}^2\ge 0$,
		\[
		\btr{\Sigma}\norm{\mathcal{H}^2-\fint\mathcal{H}^2}^2_{L^2(\Sigma)}\le 2\btr{\Sigma}\newnorm{\accentset{\circ}{A}}_{L^2(\Sigma)}^2,
		\]
		with equality if and only if $\Sigma$ is an STCMC surface.
	\end{prop}
	\begin{bem}\label{bem_tracefreeA3} The proof is motivated by the work of DeLellis--Topping \cite{Delellistopping} on an almost-Schur lemma. In fact, one may view the above as a generalization of these inequalities to $n=2$ in the case of surfaces of genus $0$.
	\end{bem}
	\begin{proof}
		Using Proposition \ref{prop_codazziminkowski2} we conclude
		\begin{align}\label{eq_almostschur1}
		\dive \accentset{\circ}{A}=\frac{1}{2}\d \mathcal{H}^2.
		\end{align}
		Now consider the elliptic equation
		\begin{align}\label{eq_almostschur2}
		\Delta f=\mathcal{H}^2-\fint\mathcal{H}^2.
		\end{align}
		Then \eqref{eq_almostschur2} has a unique solution $f$ such that $\int f=0$, cf. \cite[Chapter 2.3]{sauter}. Integration by parts using \eqref{eq_almostschur1} and \eqref{eq_almostschur2} gives
		\begin{align*}
		\int \left(\mathcal{H}^2-\fint\mathcal{H}^2\right)^2&=-\int\spann{\nabla\mathcal{H}^2,\nabla f}\\
		&=-2\int\dive\accentset{\circ}{A}(\nabla f)\\
		&=2\int \spann{\accentset{\circ}{A},\Hess f}\\
		&=2\int \spann{\accentset{\circ}{A},\accentset{\circ}{\Hess }f}\\
		&\le 2\,\newnorm{\accentset{\circ}{A}}_{L^2}\cdot\newnorm{\accentset{\circ}{\Hess }f}_{L^2}
		\end{align*}
		Using the Bochner formula \eqref{eq_bochner} and recalling the fact that $\Ric_\gamma=\frac{1}{4}\mathcal{H}^2\gamma$, we find that
		\begin{align*}
		\newnorm{\accentset{\circ}{\Hess }f}_{L^2}^2=\int \btr{{\Hess }f}^2-\frac{1}{2}(\Delta f)^2
		=\int\frac{1}{2}(\Delta f)^2-\frac{1}{4}\mathcal{H}^2\btr{\nabla f}^2
		\le \frac{1}{2}\int \left(\mathcal{H}^2-\fint\mathcal{H}^2\right)^2.
		\end{align*}
		Combining both inequalities implies the estimate. It remains to show that $\mathcal{H}^2=\fint\mathcal{H}^2$ if equality is achieved. In the case of equality, the estimate involving the Bochner formula gives $f$ constant on $\operatorname{supp}\mathcal{H}^2$. Thus, $\mathcal{H}^2=\fint\mathcal{H}^2=\frac{4}{r^2}$ on $\operatorname{supp}\mathcal{H}^2$. By continuity of $\mathcal{H}^2$, we conclude that either $\operatorname{supp}\mathcal{H}^2=\emptyset$ or $\operatorname{supp}\mathcal{H}^2=\Sigma$. As $\operatorname{supp}\mathcal{H}^2\not=\emptyset$ by Gauss--Bonnet, we have $\mathcal{H}^2=\fint\mathcal{H}^2$ everywhere.
	\end{proof}
	We obtain the following, optimal estimate.
	\begin{thm}\label{thm_tracefreeA4}
		Let $(\Sigma,\gamma)$ be a spacelike cross section of the standard lightcone in the $3+1$-dimensional Minkowski spacetime with $\mathcal{H}^2\ge 0$. Then, we have that
		\[
		\btr{\Sigma}\int_\Sigma \btr{A-\frac{\fint\mathcal{H}^2}{2}\gamma}^2\le 2\btr{\Sigma} \int_\Sigma\newbtr{\accentset{\circ}{A}}^2,
		\]
		and equality holds if and only if $\Sigma$ is a surface of constant spacetime mean curvature. Moreover, the constant $2$ is optimal.
	\end{thm}
	\begin{proof}
		As before, the estimate is equivalent to Proposition \ref{thm_tracefreeA3}. It remains to show that the constant $C=2$ is optimal. We will argue similar as in \cite[Section 3]{Delellistopping}.
		Let $f$ be a non-constant eigenfunction of the Laplace operator on $(\Sbb^2,\d\Omega^2)$. In particular,
		\begin{align}\label{eq_meanvaluezero}
			\int_{\Sbb^2}f=0
		\end{align}
		and $\Delta_{\Sbb^2} f=-\lambda f$, where $\lambda=-l(l+1)$ for some $l\ge 1$. Consider the functions 
		\[
			\omega_s:=\frac{1}{1+sf},
		\]
		then for some sufficiently small $\varepsilon>0$, the spacelike cross sections $\Sigma_{\omega_s}$ are well-defined ($\omega_s>0$) and have strictly positive spacetime mean curvature $\mathcal{H}^2_s$ for all $s\in(-\varepsilon,\varepsilon)$. By Remark \ref{bem_codazzi} and the well-known formula for the scalar curvature, cf. \cite[(4.18)]{chenwang}, we recall that
		\begin{align}
		\begin{split}\label{eq_identitiesVs}
			\accentset{\circ}{A}&=4v_s^{-1}\accentset{\circ}{\Hess}_{\Sbb^2}v_s,\\
			\mathcal{H}^2&=4v_s^2-\btr{{}^{\Sbb^2}\!\nabla v_s}^2+4v_s\Delta_{\Sbb^2}v_s
		\end{split}
		\end{align}
		for $v_s:=w_s^{-1}=1+sf$, where ${}^{\Sbb^2}\!\nabla$, $\Hess_{\Sbb^2}$, $\Delta_{\Sbb^2}$ denote the gradient, Hessian, and Laplace operator with respect to the round metric $\d\Omega^2$, respectively. 
		
		Let $C>0$ be a fixed constant. We consider the function $F_C\colon (-\varepsilon,\varepsilon)\to\R$ defined as
		\[
			F_C(t):=C\cdot\btr{\Sigma_{\omega_s}}\int_{\Sigma_{\omega_s}}\newbtr{\accentset{\circ}{A}}^2+256\pi^2-\btr{\Sigma_{\omega_s}}\int_{\Sigma_{\omega_s}}(\mathcal{H}^2)^2.
		\]
		To show that $2$ is the optimal constant, it suffices to show that for every $0<C<2$ there exists some choice of $f$ such that $F_C(t_0)<0$ for some $t_0\in(-\varepsilon,\varepsilon)$. To this end, note that $F_C(0)=0$ and we compute the first and second derivative of $F_C$. Using \eqref{eq_identitiesVs} a direct computation gives
		\begin{align*}
		\btr{\Sigma_{\omega_s}}&=\int_{\Sbb^2}\omega_s^2,\\
			\int_{\Sigma_{\omega_s}}\newbtr{\accentset{\circ}{A}}^2&=16\int_{\Sbb^2}\newbtr{\accentset{\circ}{\Hess}_{\Sbb^2}v_s}^2,\\
			\int_{\Sigma_{\omega_s}}(\mathcal{H}^2)^2&=16\int_{\Sbb^2}\left(v_s-v_s^{-1}\btr{{}^{\Sbb^2}\!\nabla v_s}^2+\Delta_{\Sbb^2}v_s\right)^2.
		\end{align*}
		Although one can derive the first and second variation of $A$, $\mathcal{H}^2$ from the propagation equations, cf. \cite[Lemma 8]{wolff1}, the explicit identities with respect to the round metric above allow us to conclude by direct computation that
		\[
			\left.\frac{\d}{\d s}\btr{\Sigma_{\omega_s}}\right\vert_{s=0}=\left.\frac{\d}{\d s}
			\int_{\Sigma_{\omega_s}}\newbtr{\accentset{\circ}{A}}^2\right\vert_{s=0}=\left.\frac{\d}{\d s}	\int_{\Sigma_{\omega_s}}(\mathcal{H}^2)^2\right\vert_{s=0}=0,
		\]
		where we used \eqref{eq_meanvaluezero} and the fact that $v_0\equiv 1$. Similarly, direct computation gives
		\begin{align*}
			\left.\frac{\d^2}{\d s^2}\int_{\Sbb^2}\omega_s^2\right\vert_{s=0}&=4\int_{\Sbb^2}f^2,\\
			\left.\frac{\d^2}{\d s^2}\int_{\Sbb^2}\left(v_s-v_s^{-1}\btr{{}^{\Sbb^2}\!\nabla v_s}^2+\Delta_{\Sbb^2}v_s\right)^2\right\vert_{s=0}&=2\int_{\Sbb^2}\left(f^2+(\Delta_{\Sbb^2}f)^2\right)-8\int_{\Sbb^2}\btr{{}^{\Sbb^2}\!\nabla f}^2,\\
			\left.\frac{\d^2}{\d s^2}\int_{\Sbb^2}\newbtr{\accentset{\circ}{\Hess}_{\Sbb^2}v_s}^2\right\vert_{s=0}&=2\int_{\Sbb^2}\newbtr{\accentset{\circ}{\Hess}_{\Sbb^2}f}^2=\int_{\Sbb^2}(\Delta_{\Sbb^2}f)^2-2\btr{{}^{\Sbb^2}\!\nabla f}^2,
		\end{align*}
		where we used the Bochner formula \eqref{eq_bochner} to derive the second equality in the last line. We conclude that $F'_C(0)=0$ and
		\begin{align*}
			F_C''(t)&=64\pi\left(\int_{\Sbb^2}(C-2)(\Delta_{\Sbb^2}f)^2+(8-2C)\btr{{}^{\Sbb^2}\!\nabla f}^2-6f^2\right)\\
			&=64\pi\int_{\Sbb^2}f^2\left((C-2)\btr{\lambda}^2+(8-2C)\btr{\lambda}-6\right)<0
		\end{align*}
		for any $C<2$ and eigenvalue $\lambda$ sufficiently big. The claim then follows by Taylor approximation.
	\end{proof}
	\begin{bem}
		Observe that the polynomial $p_C(x)=(C-2)x^2+(8-2C)x-6$ is strictly positive on $[2,\infty)$ for all $C\ge 2$, which is consistent with the estimate proven in Theorem \ref{thm_tracefreeA4}.
	\end{bem}
\section{An elliptic estimate on $\mathbb{S}^2$ under a balancing condition}\label{sec_elliptic}
	We dedicate most of this subsection to prove the following proposition:
	\begin{prop}\label{prop_w22}
		Let $(\Sigma,\gamma)$ be a conformally round surface with conformal factor $\omega$ such that $\btr{\Sigma}=4\pi$ and $C_0^{-1}\le \omega\le C_0$, $C_0>1$. We further assume that $\Sigma$ satisfies the balancing condition
		\begin{align}\label{eq_balancing}
		\int_{\mathbb{S}^2}f_i\omega^3=0\text{  for all }i=1,2,3,
		\end{align}
		where $f_i$ denote first spherical harmonics. Then there exists an $\varepsilon>0$ only depending on $C_0$, such that if $\norm{\mathcal{K}-1}_{L^2(\mathbb{S}^2)}\le\varepsilon$, we have
		\[
		\norm{\omega-1}_{W^{2,2}(\mathbb{S}^2)}\le C\cdot\norm{\mathcal{K}-1}_{L^2(\mathbb{S}^2)},
		\]
		where $C$ is a constant only depending on $C_0$.
	\end{prop}
	Estimates in similar spirit are well-known, see e.g. \cite{chang, klainermanszeftel, shiwangwu}. In fact, we will closely follow the strategy outlined by Shi--Wang--Wu \cite{shiwangwu}, replacing the sup-estimates with the application of the H\"older inequality. Compared to \cite{shiwangwu} we assume additional uniform sup-bounds on $\omega$, and we note the different balancing condition \eqref{eq_balancing} with respect to the usual one found in the literature, see e.g. \cite{chang, changyang, klainermanszeftel,onofri, shiwangwu}, which here is motivated by the definition of the associated $4$-vector, cf. Section \ref{sec_as4vector}.
	
	We first establish several auxiliary lemmas for the smooth function $u:= \ln(\omega)$ on $\mathbb{S}^2$.
	\begin{lem}\label{lem_w22_1}
		Under the assumptions of Proposition \ref{prop_w22} and for any $\eta>0$, there exists a $\delta>0$ such that if
		\[
		\norm{\mathcal{K}-1}_{L^2(\mathbb{S}^2)}\le \delta,
		\]
		then $\btr{u}_\infty\le \eta$.
	\end{lem}
	\begin{proof}
		Assume that this is false. Then there exists $\eta_0>0$ and a sequence $(u_k)_{k\in\mathbb{N}}$ of smooth functions corresponding to conformally round surfaces $\Sigma_k$ with $\btr{\Sigma_k}=4\pi$, $w_k=e^{u_k}$ satisfies \eqref{eq_balancing}, and $0<\eta_0\le \btr{u_k}_\infty\le C_1$, $C_1:=\ln(C_0)$ for all $k$ with 
		\[
		\norm{\mathcal{K}_k-1}_{L^2(\mathbb{S}^2)}\to 0
		\]
		as $k\to\infty$. Note that with respect to $u_k$ the Gauss curvature $\mathcal{K}_k$ satisfies
		\[
		\Delta_{\mathbb{S}^2}u_k=1-e^{2u_k}\mathcal{K}_k,
		\]
		and in particular, by the above properties on the sequence,
		\[
		\norm{\Delta_{\mathbb{S}^2}u_k}_{L^2(\mathbb{S}^2)}\le C
		\]
		for some constant independent of $k$. By the uniform estimate on $\btr{u_k}_\infty$, standard elliptic estimates yield that $\norm{u_k}_{W^{2,2}(\mathbb{S}^2)}$ is bounded independent of $k$, cf. \cite[Appendix H]{besse}. As $W^{2,2}$ embeds compactly into $W^{1,2}$ and $C^{0,\alpha}$ for some $\alpha\in(0,1)$ on $\mathbb{S}^2$, cf. \cite[Appendix C]{besse}, there exists a subsequence $(u_k)$ such that $u_k\to u_\infty\in W^{1,2}\cap C^0$ (after possibly taking successive subsequences) with convergence both in $C^0$ and $W^{1,2}$. In particular, $u_\infty$ satisfies $\eta_0\le \btr{u_\infty}_\infty\le C_1$,
		\[
		\int_{\mathbb{S}^2}f_ie^{3u_\infty}=0,
		\]
		and weakly solves the equation 
		\[
		\Delta_{\mathbb{S}^2}u_\infty=1-e^{2u_\infty}.
		\]
		By appealing to standard arguments in regularity theory, we see that $u_\infty$ is indeed smooth and solves the above equation in the strong sense. In particular $(\mathbb{S}^2,e^{2u_\infty})$ is a smooth conformally round surface with area $4\pi$ and constant Gauss curvature everywhere equal to $1$. However, this gives an immediate contradiction as $u_\infty=0$ is the unique solution under the balancing condition \eqref{eq_balancing}, cf. Remark \ref{bem_codazzi} and Remark \ref{bem_conformalfactorLorentztzransformation} (ii).
	\end{proof}
	\begin{lem}\label{lem_w22_2}
		For any bounded, measurable function $u$ on $\mathbb{S}^2$ with $\btr{u}\le \ln(C_0)$, $C_0>1$ and $k\in \mathbb{N}$, we have
		\[
		\norm{1+ku-e^{ku}}_{L^2(\mathbb{S}^2)}\le C\btr{u}_\infty \norm{u}_{L^2(\mathbb{S}^2)}
		\]
		for some constant $C$ only depending on $C_0$ and $k$.
	\end{lem}
	\begin{proof}
		For any $p\in\mathbb{S}^2$ with $u(p)\not=0$, we find that
		\begin{align*}
		\btr{1+ku(p)-e^{ku(p)}}\le C \btr{u(p)}^2,
		\end{align*}
		using the mean value theorem. As the inequality is trivially satisfied whenever $u(p)=0$, we conclude
		\[
		\norm{1+ku-e^{ku}}_{L^2(\mathbb{S}^2)}\le C\norm{u}_{L^4(\mathbb{S}^2)}^{2}\le C\btr{u}_\infty \norm{u}_{L^2(\mathbb{S}^2)}
		\]
	\end{proof}
	We now prove Proposition \ref{prop_w22}:
	\begin{proof}[Proof of Proposition \ref{prop_w22}]
		We decompose $u$ into $u=u_0+u_1+u_2$ with respect to an $L^2$-orthonormal basis of spherical harmonics on $\mathbb{S}^2$, where
		\[
		u_0:=\frac{1}{4\pi}\int_{\mathbb{S}^2}u,
		\]
		$u_1:=a^1f_1+a^2f_2+a^3f_3$ with
		\[
		a^i:=\frac{3}{4\pi}\int_{\mathbb{S}^2}u\cdot f_i.
		\]
		Note that $u_2:=u-u_0-u_1$ is perpendicular to $u_0,u_1$ in $L^2$. We first aim to bound the $L^2$-norm of $u$. As
		\[
		\norm{u}_{L^2(\mathbb{S}^2)}\le \left(\btr{u_0}+\sum\limits_{i=1}^3\newbtr{a^i}+\norm{u_2}_{L^2(\mathbb{S}^2)}\right),
		\]
		it suffices to bound the individual terms. In the following, $C$ will always denote a constant only depending on $C_0$ and universal constants that may vary from line to line. Direct estimation gives
		\begin{align*}
		\btr{u_0}=\frac{1}{8\pi}\btr{\int 2u}
		&=\frac{1}{8\pi}\btr{\int 2u+1-\mathcal{K}e^{2u}}\\
		&=\frac{1}{8\pi}\btr{\int(2u+1-e^{2u})+(1-\mathcal{K})e^{2u}}\\
		&\le C\left(\norm{1+2u-e^{2u}}_{L^2(\mathbb{S}^2)}+\norm{1-\mathcal{K}}_{L^2(\mathbb{S}^2)}\right),
		\end{align*}
		and
		\begin{align*}
		4\pi \newbtr{a^i}=\btr{\int 3u f_i}=\btr{\int (3u+1-e^{3u})f_i}\le C\norm{1+3u-e^{3u}}_{L^2(\mathbb{S}^2)},
		\end{align*}
		where we used the Gauss--Bonnet theorem and the balancing condition \eqref{eq_balancing} in the first and second computation, respectively, and the H\"older inequality in both cases in the last line. It remains to bound $\norm{u_2}_{L^2(\mathbb{S}^2)}$.
		
		We note that by our choice of decomposition, we find that for the operator $L$ defined as
		\[
		L(v):=\Delta v+2v
		\]
		$u_2$ satisfies
		\begin{align*}
		L(u_2)=L(u)-2u_0=(1+2u-e^{2u})+(1-\mathcal{K})e^{2u}-2u_0,
		\end{align*}
		where we used the explicit formula for the Gauss curvature. Multiplying the equation by $u_2$, integration by parts yields
		\[
		\int\btr{\nabla u_2}^2-2u_2^2=-\int u_2L(u_2)\le C\norm{u_2}_{L^2(\mathbb{S}^2)}\left(\norm{1+2u-e^{2u}}_{L^2(\mathbb{S}^2)}+\norm{1-\mathcal{K}}_{L^2(\mathbb{S}^2)}\right),
		\]
		where we used the H\"older inequality and the already established bound on $u_0$. Moreover, we also know that by our choice of decomposition
		\[
		u_2=\sum\limits_{l=2}^\infty\sum\limits_{k=-l}^l a_{l,k}Y_l^k,
		\]
		where $Y_l^k$ denote spherical harmonics with eigenvalues $\lambda_l=-l(l+1)\le - 6$. Hence, partial integration yields
		\[
		\int\btr{\nabla u_2}^2=-\int\sum\limits_{l=2}^\infty\sum\limits_{k=-l}^l\lambda_l a_{l,k}^2(Y_l^k)^2\ge 6\norm{u_2}_{L^2(\mathbb{S}^2)}^2,		
		\]
		and we conclude that
		\[
		4\norm{u_2}_{L^2(\mathbb{S}^2)}^2\le C\norm{u_2}_{L^2(\mathbb{S}^2)}\left(\norm{1+2u-e^{2u}}_{L^2(\mathbb{S}^2)}+\norm{1-\mathcal{K}}_{L^2(\mathbb{S}^2)}\right).
		\]
		Using the Peter--Paul inequality, we obtain
		\[
		\norm{u_2}_{L(\mathbb{S}^2)}\le C\left(\norm{1+2u-e^{2u}}_{L^2(\mathbb{S}^2)}+\norm{1-\mathcal{K}}_{L^2(\mathbb{S}^2)}\right)
		\]
		after taking a square root. Combining the above estimates, we find that
		\begin{align*}
		\norm{u}_{L^2(\mathbb{S}^2)}&\le C\left(\norm{1-\mathcal{K}}_{L^2(\mathbb{S}^2)}+\norm{1+2u-e^{2u}}_{L^2(\mathbb{S}^2)}+\norm{1+3u-e^{3u}}_{L^2(\mathbb{S}^2)}\right)\\
		&\le C\left(\norm{1-\mathcal{K}}_{L^2(\mathbb{S}^2)}+\btr{u}_\infty\norm{u}_{L^2(\mathbb{S}^2)}\right),
		\end{align*}
		using Lemma \ref{lem_w22_2}. Let $\eta>0$ such that $\eta C\le \frac{1}{2}$. Then by Lemma \ref{lem_w22_1} there exists $\delta(\eta)$ (only depending on $C_0$) such that
		\[
		\norm{u}_{L(\mathbb{S}^2)}\le 2C\norm{1-\mathcal{K}}_{L^2(\mathbb{S}^2)}
		\]
		if $\norm{1-\mathcal{K}}_{L^2(\mathbb{S}^2)}\le \delta(\eta)$. As
		\[
		\norm{\Delta u}_{L^2(\mathbb{S}^2)}=\norm{1-\mathcal{K}e^{2u}}_{L^2(\mathbb{S}^2)}\le C\left(\norm{u}_{L^2(\mathbb{S}^2)}+\norm{1-\mathcal{K}}_{L^2(\mathbb{S}^2)}\right),
		\]
		it follows by standard elliptic estimates, cf. \cite[Appendix H]{besse}, that
		\[
		\norm{u}_{W^{2,2}(\mathbb{S}^2)}\le C\norm{1-\mathcal{K}}_{L^2(\mathbb{S}^2)}
		\]
		if $\norm{1-\mathcal{K}}_{L^2(\mathbb{S}^2)}\le \delta(\eta)$. The claim then follows with $\varepsilon:=\delta(\eta)$ and noting that \linebreak $\norm{\omega-1}_{W^{2,2}(\mathbb{S}^2)}\le C\norm{u}_{W^{2,2}(\mathbb{S}^2)}$.
	\end{proof}
	Before stating the main result in the next subsection, we give a partial result formulated as an intrinsic result on conformally round surfaces.
	\begin{kor}\label{kor_W22}
		Let $(\Sigma,\gamma)$ be a conformally round surface with non-negative scalar curvature $\operatorname{R}\ge 0$ and area radius $r$. Assume that $C_0^{-1}r\le \omega\le C_0r$ for some constant $C_0$, and that the balancing condition \eqref{eq_balancing} holds. Then there exits $\varepsilon>0$ only depending on $C_0$ such that if
		\[
		\btr{\Sigma}\cdot \newnorm{\accentset{\circ}{\Hess}_{\mathbb{S}^2}v}_{L^2(\mathbb{S}^2)}^2\le \varepsilon
		\]
		for $v:=\frac{1}{\omega}$, then
		\[
		\norm{\omega-r}_{W^{2,2}(\mathbb{S}^2)}\le C\btr{\Sigma}\cdot \newnorm{\accentset{\circ}{\Hess}_{\mathbb{S}^2}v}_{L^2(\mathbb{S}^2)}
		\]
		for a constant $C$ only depending on $C_0$.
	\end{kor}
	\begin{proof}
		By the Gauss equation, $\Sigma$ has $\mathcal{H}^2\ge 0$ as a spacelike cross section of the standard lightcone. Note further that
		\[
		\newnorm{\accentset{\circ}{A}}_{L^2(\Sigma)}=4\newnorm{\accentset{\circ}{\Hess}_{\mathbb{S}^2}v}_{L^2(\mathbb{S}^2)}^2,
		\]
		cf. Remark \ref{bem_codazzi},and that for the rescaled surface with conformal factor $\widetilde{\omega}=\frac{\omega}{r}$ it holds that  
		\begin{align*}
		\norm{\widetilde{\mathcal{K}}-1}^2_{L^2(\mathbb{S}^2)}\le C(C_0)\btr{\Sigma}\norm{\mathcal{H}^2-\fint\mathcal{H}^2}^2_{L^2(\Sigma)},
		\end{align*}
		Then estimate immediately follows from Theorem \ref{thm_tracefreeA3} and Proposition \ref{prop_w22}, after multiplying by the area radius $r$.
	\end{proof}\pagebreak
\section{A De\,Lellis--M\"uller type estimate on the lightcone}\label{sec_main}
	We now state the main result.
	\begin{thm}\label{thm_delellismueller}
		Let $(\Sigma,\gamma)$ be a spacelike cross section of the future-pointing standard lightcone in the $1+3$-Minkowski spacetime with $\mathcal{H}^2\ge 0$, and associated $4$-vector ${\textbf{Z}}={\textbf{Z}}(\Sigma)$. Assume further that $\Sigma$ is $\kappa$-bounded for some fixed $\kappa>0$. Then there exists $\varepsilon>0$ only depending on $\kappa$ such that if
		\[
		\btr{\Sigma}\cdot\newnorm{\accentset{\circ}{A}}^2_{L^2(\Sigma)}\le \varepsilon
		\]
		then
		\[
		\newnorm{\omega-\omega_{{\textbf{Z}}}}_{W^{2,2}(\mathbb{S}^2)}\le C(\kappa,\textbf{Z})\cdot\btr{\Sigma}\newnorm{\accentset{\circ}{A}}_{L^2(\Sigma)},
		\]
		where $C(\kappa,\Lambda)$ is a constant only depending on $\kappa$ and $\textbf{Z}$.
	\end{thm}
	\begin{proof}
		By Lemma \ref{lem_kappabounded2} there exists $\Lambda\in \operatorname{SO}^+(1,3)$ (uniquely determined by $\textbf{Z}$ up to isometries on $\Sbb^2$) such that $\omega_{\Lambda^{-1}(\textbf{Z})}=r_{\textbf{Z}}=\sqrt{-\eta(\textbf{Z},\textbf{Z})}$, as this is equivalent to
		\[
		\int_{\Sbb^2} f_i\omega^3=0\text{ for all }i=1,2,3.
		\] 
		In particular, the components of $\Lambda$ are uniquely determined by $\textbf{Z}$ (up to a choice of rotation).
		By Remark \ref{bem_Lorentztransform}, Lemma \ref{lem_kappabounded2}, Theorem \ref{thm_tracefreeA3}, and the Gauss equation \eqref{eq_gaußcurvature}, we find for the rescaled surface corresponding to $\widetilde{\omega}_{\Lambda^{-1}}:=r^{-1}\omega_{\Lambda^{-1}}$ that
		\begin{align*}
		\newnorm{\widetilde{\mathcal{K}}_{\Lambda^{-1}}-1}^2_{L(\mathbb{S}^2)}&\le C(\kappa)\btr{\Lambda^{-1}(\Sigma)}\norm{\mathcal{H}_{\Lambda^{-1}}^2-\fint\mathcal{H}_{\Lambda^{-1}}^2}^2_{L^2(\Lambda^{-1}(\Sigma))}
		\\
		&=C(\kappa)\btr{\Sigma}\norm{\mathcal{H}^2-\fint\mathcal{H}^2}^2_{L^2(\Sigma)}
		\\
		&\le 2C(\kappa)\btr{\Sigma}\newnorm{\accentset{\circ}{A}}^2_{L^2(\Sigma)}.
		\end{align*}
		Hence, for $\varepsilon$ sufficiently small, we may use Proposition \ref{prop_w22} as $\widetilde{\omega}_{\Lambda^{-1}}$ has area $4\pi$ and satisfies the balancing condition \eqref{eq_balancing}. Multiplying by $r$, we find
		\[
		\newnorm{\omega_{\Lambda^{-1}}-r}_{W^{2,2}(\mathbb{S}^2)}\le C(\kappa)\btr{\Sigma}\newnorm{\accentset{\circ}{A}}_{L^2(\Sigma)}
		\]
		Now notice that by the definition of $\omega_{\textbf{Z}}$, $r_{\textbf{Z}}$ and Lemma \ref{lem_conformalfactorLorentztzransformation}, we find that
		\[
		\omega-\frac{r}{r_{\textbf{Z}}}\omega_{\textbf{Z}}=\frac{\omega_{\Lambda^{-1}}\circ\Phi_\Lambda-r}{\sqrt{1+\newbtr{\vec{a}^2}}-\vec{a}_if_i},
		\]
		where $\Phi_\Lambda$ and $\vec{a}$ are uniquely determined by $\textbf{Z}$ up to a rotation. Hence
		\[
		\newnorm{\omega-\frac{r}{r_{\textbf{Z}}}\omega_{\textbf{Z}}}_{W^{2,2}(\mathbb{S}^2)}\le C(\textbf{Z})\newnorm{\omega_{\Lambda^{-1}}-r}_{W^{2,2}(\mathbb{S}^2)}\le  C(\kappa,\textbf{Z})\btr{\Sigma}\newnorm{\accentset{\circ}{A}}_{L^2(\Sigma)},
		\]
		where the constant $C(\textbf{Z})$ does not depend on the choice of rotation as they act as isometries on $\Sbb^2$.
		On the other hand, note that $W^{2,2}$ embedds compactly into $C^0$ on $\mathbb{S}^2$. Thus, we
		find that $\btr{\omega_{\Lambda^{-1}}-r}_\infty\le C(\kappa)\btr{\Sigma}\newnorm{\accentset{\circ}{A}}_{L^2(\Sigma)}$. Notice further that
		\[
		r_{\textbf{Z}}=r_{\textbf{Z}(\Lambda^{-1}(\Sigma))}=\frac{1}{\btr{\Lambda^{-1}(\Sigma)}}\int_{\mathbb{S}^2}\omega_{\Lambda^{-1}}^{3}\le r+\btr{\omega_{\Lambda^{-1}}-r}_\infty.
		\]
		By Remark \ref{bem_kappabounded}, we conclude that
		\[
		\btr{r_{\textbf{Z}}-r}\le C(\kappa)\btr{\Sigma}\newnorm{\accentset{\circ}{A}}_{L^2(\Sigma)}.
		\]
		Hence,
		\[
		\newnorm{\frac{r}{r_{\textbf{Z}}}\omega_{\textbf{Z}}-\omega_{\textbf{Z}}}_{W^{2,2}(\mathbb{S}^2)}=\btr{r-r_{\textbf{Z}}}\cdot\norm{\frac{\omega_{\textbf{Z}}}{r_{\textbf{Z}}}}_{W^{2,2}(\mathbb{S}^2)}\le C(\textbf{Z})\btr{r_{\textbf{Z}}-r}\le C(\kappa,\textbf{Z})\btr{\Sigma}\newnorm{\accentset{\circ}{A}}_{L^2(\Sigma)},
		\]
		where the constant $C(\textbf{Z})$ again does not depend on the choice of rotation.
		The claim then follows from the triangle inequality.
	\end{proof}
	\begin{bem}\label{bem_delellismueller}\,
	\begin{enumerate}
		\item[(i)] We note that the constant $C(\kappa,\textbf{Z})$ in fact only depends on $\kappa$ and $\frac{\textbf{Z}}{r_\textbf{Z}}$. As 
		\[
		\textbf{Z}=r_{\textbf{Z}}(\sqrt{1+\newbtr{\vec{a}}^2},\vec{a})
		\] 
		for some $\vec{a}\in\R^3$, we may take $\Lambda=\Lambda_{\vec{a}}$ above. By this choice, it is straightforward to check that one can moreover choose $C(\kappa,\textbf{Z})$ to continuously depend on $\frac{\textbf{Z}}{r_\textbf{Z}}$.
		\item[(ii)] Note that as an intermediate step, we have proven that
		\begin{align}\label{eq_delellismueller2}
		\norm{\frac{\omega}{r}-\frac{\omega_{\textbf{Z}}}{r_{\textbf{Z}}}}^2_{W^{2,2}(\mathbb{S}^2)}\le  C(\kappa,\textbf{Z})\btr{\Sigma}\newnorm{\accentset{\circ}{A}}_{L^2(\Sigma)}^2,
		\end{align}
		which allows for a direct comparison of the spacelike cross section $\Sigma$ and the STCMC surface of reference both rescaled to have area radius $1$. 
	\end{enumerate}
	\end{bem}
	Using the observations in Remark \ref{bem_delellismueller}, we quickly establish the following conclusion from Theorem \ref{thm_delellismueller}.
	\begin{thm}\label{thm_delellismueller2}
		Let $(\Sigma_k)_{k\in\N}$ be a sequence of spacelike cross sections of the standard lightcone in the $1+3$-Minkowski spacetime with $\mathcal{H}^2_k\ge 0$ with conformal factor $\omega_k$, area radius $r_k$ and associated $4$-vector ${\textbf{Z}}_k$. Assume further that $\Sigma_k$ is $\kappa$-bounded for all $k$ with $\kappa$ independent of $k$, there exists $\delta>0$ such that
		\begin{align}\label{eq_deltatimelike}
		\left({\textbf{Z}}^0_k\right)^2\ge (1+\delta)\left(\left({\textbf{Z}}^1_k\right)^2+\left({\textbf{Z}}^2_k\right)^2+\left({\textbf{Z}}^3_k\right)^2\right)
		\end{align}
		for all $k$, and
		\begin{align}\label{eq_Akconvergence}
		\btr{\Sigma_k}\cdot\newnorm{\accentset{\circ}{A}_k}^2_{L^2(\Sigma_k)}\to 0
		\end{align}
		as $k\to\infty$. Then there exists a subsequence $(k_l)_{l\in\N}$ and a surface of constant spacetime mean curvature $\widetilde\Sigma$ with conformal factor $\widetilde{\omega}$ and area radius $1$ such that
		\[
		\frac{\omega_{k_l}}{r_{k_l}}\to\widetilde{\omega}\text{ in }W^{2,2}(\mathbb{S}^2),
		\]
		and 
		\[
		\frac{{\textbf{Z}}_{k_l}}{r_{{\textbf{Z}}_{k_l}}}\to{\textbf{Z}}(\widetilde\Sigma),
		\]
		where $r_{{\textbf{Z}}_{k_l}}=\sqrt{-\eta({\textbf{Z}}_{k_l},{\textbf{Z}}_{k_l})}$, and $\eta({\textbf{Z}}(\widetilde\Sigma),{\textbf{Z}}(\widetilde\Sigma))=-1$.
	\end{thm}
	\begin{bem}
		As $W^{2,2}$ embeds compactly into $C^0$ on $\Sbb^2$, cf. \cite[Appendix C]{besse}, the rescaled conformal factors converge in $C^0$. Note that for any uniformly converging sequence of spacelike cross sections, we have convergence of the associated $4$-vectors to the associated $4$-vector of the limiting surface, which is easy to verify from the definition of ${\textbf{Z}}$, cf. Section \ref{sec_as4vector}. In particular, the sequence is $\kappa$-bounded with $\kappa$ independent of $k$ and satisfies \eqref{eq_deltatimelike} for $k$ sufficiently large. Hence, these are necessary conditions for the conclusions of Theorem \ref{thm_delellismueller2} to hold.
	
	As a  counterexample, the STCMC surfaces corresponding to the conformal factors
	\[
	\omega_k=\frac{1}{\sqrt{1+k^2}-k\cos\theta}
	\]
	have area radius $1$, $\mathcal{H}^2_k=4$, are $\kappa$-bounded for any $\kappa>0$ independent of $k$, and $\accentset{\circ}{A}_k=0$ for all $k$. However, the sequence $\omega_k$ does not converge to a limit, even in $C^0$, due to the non-compactness of the restricted Lorentz group $\operatorname{SO}^+(1,3)$. This issue is precisely avoided by \eqref{eq_deltatimelike}, as the renormalized associated $4$-vectors will be restricted to a compact subset of the hyperboloid $\{p\in \R^{1,3}\colon \eta(p,p)=-1,\, p^0>0\}$.
	\end{bem}
	\begin{proof}[Proof of Theoem \ref{thm_delellismueller2}]
		We define the sequence 
		\[
		{z}_k:=\frac{{\textbf{Z}}_{k_l}}{r_{{\textbf{Z}}_{k_l}}}.
		\]
		In particular $\eta({z}_k,{z}_k)=-1$ and as \eqref{eq_deltatimelike} is preserved under rescaling, we see that
		\[
		(1+\delta)\left(({z}^1_k)^2+({z}^2_k)^2+({z}^3_k)^2\right)\le ({z}^0_k)^2=({z}^1_k)^2+({z}^2_k)^2+({z}^3_k)^2+1.
		\]
		Hence
		\[
		{z}_k\in C_\delta:=\left\{p\in \R^{1,3}\colon (p^1,p^2,p^3)\in \overline{B}_{\frac{1}{\delta}}(0),\, p^0=\sqrt{1+(p^1)^2+(p^2)^2+(p^3)^2}\right\}
		\]
		for all $k$, where $\overline{B}_{\frac{1}{\delta}}(0)$ denotes the closed ball of radius $\frac{1}{\delta}$ in $\R^3$ centered around the origin. As the set is compact, there exists a subsequence $(k_l)_{l\in\N}$ such that
		\[
		{z}_{k_l}\to {z}\in C_\delta.
		\]
		In particular, ${z}$ is timelike, future-pointing with $\eta({z},{z})=-1$. Let $\widetilde{\Sigma}$ be the STCMC surface with conformal factor $\widetilde{\omega}:=\omega_{{z}}$ as defined in Remark \ref{bem_conformalfactorLorentztzransformation} (ii). Then $\widetilde{\Sigma}$ has area radius $\widetilde{r}=r_{{z}}:=\sqrt{-\eta({z},{z})}=1$. Moreover, as ${z}_{k_l}\to {z}$, we see by the explicit definition \eqref{eq_ZvectorSTCMC} that
		\begin{align}\label{eq_THMdelellismueller1}
		\omega_{{{z}_{k_l}}}\to \widetilde{\omega}
		\end{align}
		in $C^2$ (in fact $C^m$ for any fixed $m$). It only remains to show the $W^{2,2}$ convergence.
		
		To this end, recall that $\omega_{{{z}_{k_l}}}=\frac{1}{r_{{\textbf{Z}}_{k_l}}}\omega_{{\textbf{Z}}_{k_l}}$. It follows that
		\begin{align}\label{eq_THMdelellismueller2}
		\norm{\frac{\omega_{k_l}}{r_{k_l}}-\omega_{{{z}_{k_l}}}}^2_{W^{2,2}(\mathbb{S}^2)}\le  C(\kappa,\textbf{Z}_{k_l})\btr{\Sigma_{k_l}}\newnorm{\accentset{\circ}{A}_{k_l}}_{L^2(\Sigma)}^2,
		\end{align}
		from \eqref{eq_delellismueller2}, where the constant $C(\kappa,\textbf{Z}_{k_l})$ does not depend on $r_\textbf{Z}$. As ${z}_{k_l}\to {z}$, by Remark \ref{bem_delellismueller} (i) there exists a suitable constant $C(\kappa,{z})$ only depending on $\kappa$ and ${z}$ such that
		\[
		C(\kappa,z_{k_l})\le C(\kappa,{z})
		\]	
		fo all $l$. The claim then follows from \eqref{eq_THMdelellismueller1} and \eqref{eq_THMdelellismueller2} using assumption \eqref{eq_Akconvergence} and the triangle inequality.
	\end{proof}

\section{Comments}\label{sec_discussion}
	We briefly consider the standard lightcone in the higher dimensional Minkowski spacetime $(\R^{1,n},\eta)$, $n\ge 3$. Deriving all the relevant quantities as in Section \ref{sec_prelim}, we note that possibly up to some dimension dependent constants almost all geometric properties are retained in the higher dimensional case. In particular, any spacelike cross section $\Sigma$ is conformal to the round sphere $\Sbb^{n-1}$ and the scalar second fundamental form $A=\ul{\theta}\chi$ satisfies
	\[
		\nabla_iA_{jk}=\nabla_jA_{ik}.
	\]
	Thus, Proposition \ref{thm_tracefreeA3} can be directly extended to higher dimensions in the following way:
	If $\Ric\ge 0$, then
	\[
		\norm{\mathcal{H}^2-\fint\mathcal{H}^2}^2_{L^2(\Sigma)}\le \frac{k}{k-1}\newnorm{\accentset{\circ}{A}}_{L^2(\Sigma)}^2,
	\]
	where $k=n-1$ is the dimension of $\Sigma$. As the Gauss-Equation yields that
	\begin{align*}
		\accentset{\circ}{\Ric}&=\frac{k-2}{2k}\accentset{\circ}{A},\\
		\operatorname{R}&=\frac{k-1}{k}\mathcal{H}^2,
	\end{align*}
	one can check that in fact the above estimate is indeed directly equivalent to the almost-Schur lemma by De\,Lellis--Topping \cite{Delellistopping} in the conformal class of the round sphere $\Sbb^{n-1}$ for all $n\ge 4$. In particular, the above estimate is optimal for all spacelike cross sections of the standard lightcone in $\R^{1,n}$ with $\Ric\ge 0$ for all $n\ge 3$.
	
	We further note that the corresponding estimate by De\,Lellis--M\"uller \cite{delellismueller} in $\R^3$ does not require assumptions on the sign of any intrinsic, or extrinsic curvature quantities, as assumed here, in the work of De\,Lellis--Topping \cite{Delellistopping}, or in the short proof of the same estimate in $\R^3$ by Huisken using inverse mean curvature flow, see Remark \ref{bem_tracefreeA1} (ii). To obtain this general result, De\,Lellis--M\"uller utilize deep analytical tools from partial differential equations. However, for $k\ge 5$ De\,Lellis--Topping \cite{Delellistopping} construct explicit counterexamples to show that one can not drop the assumption of non-negative Ricci curvature even for topological spheres. Whether such counterexpamples can be constructed in the conformal class of the round sphere for all $k\ge 2$, or if an estimate of the above form can be derived for all spacelike cross sections of the standard lightcone is subject of future work.

\bibliography{bib_delellismueller}

\nopagebreak
\bibliographystyle{plain}
\end{document}